\documentclass[lefttitle, 3p]{elsarticle}
\usepackage[english]{babel}
\usepackage[utf8]{inputenc}
\usepackage[T1]{fontenc}
\usepackage{url}

\makeatletter
\def\ps@pprintTitle{\let\@oddhead\@empty
  \let\@evenhead\@empty
  \def\@oddfoot{\reset@font\hfil\thepage\hfil}
  \let\@evenfoot\@oddfoot
}
\makeatother

\usepackage[deletedmarkup=xout]{changes}
\usepackage{todonotes}

\usepackage{xcolor}
\definecolor{red}{HTML}{f43545}
\definecolor{orange}{HTML}{ff8c00}
\definecolor{yellow}{HTML}{f6be00}
\definecolor{green}{HTML}{00ba71}
\definecolor{cyan}{HTML}{00b3b3}
\definecolor{indigo}{HTML}{a10b70}
\definecolor{violet}{HTML}{8a2be2}

\usepackage{algorithm, algorithmic}
\usepackage{tabularx}
\usepackage{mathtools, amsmath, amsthm, amsfonts, amssymb}
\usepackage{bm, dsfont}
\usepackage{cleveref}

\newtheorem{theorem}{Theorem}[section]
\newtheorem{lemma}[theorem]{Lemma}
\newtheorem{corollary}[theorem]{Corollary}
\newtheorem{prop}[theorem]{Proposition}
\newtheorem{remark}[theorem]{Remark}

\numberwithin{equation}{section}

\DeclareMathOperator{\Span}{span}
\DeclareMathOperator{\Tr}{trace}
\DeclareMathOperator{\conv}{conv}
\DeclareMathOperator{\sgn}{sgn}
\DeclareMathOperator{\Argmax}{arg\,max}

\renewcommand{\Im}{\operatorname{Im}}
\renewcommand{\Re}{\operatorname{Re}}

\begin{document} 

\begin{frontmatter}
    \title{Exact discretization, tight frames and recovery via $D$-optimal designs}
    
    \author[1]{Felix~Bartel}
    \ead{felix.bartel@math.tu-chemnitz.de}
    
    \author[1]{Lutz~Kämmerer}
    \ead{kaemmerer@math.tu-chemnitz.de}

    \author[1,2]{Kateryna~Pozharska}
    \ead{pozharska.k@gmail.com}
    
    \author[1]{Martin~Schäfer}
    \ead{martin.schaefer@math.tu-chemnitz.de}
    
    \author[1]{Tino~Ullrich\corref{cor1}}
    \ead{tino.ullrich@math.tu-chemnitz.de}
    
    \affiliation[1]{Faculty of Mathematics of Chemnitz University of Technology}
    \affiliation[2]{Institute of Mathematics of NAS of Ukraine}
    \cortext[cor1]{Corresponding author}
    
    \begin{abstract} $D$-optimal designs originate in statistics literature as an approach for  optimal experimental designs. In numerical analysis points and weights resulting from maximal determinants turned out to be useful for quadrature and interpolation. Also recently, two of the present authors and coauthors investigated a connection to the discretization problem for the uniform norm. Here we use this approach of maximizing the determinant of a certain Gramian matrix with respect to points and weights for the construction of tight frames and exact Marcinkiewicz-Zygmund inequalities in $L_2$. We present a direct and constructive approach resulting in a discrete measure  with at most $N \leq n^2+1$ atoms, which discretely and accurately subsamples the $L_2$-norm of complex-valued functions contained in a  given $n$-dimensional subspace. This approach can as well be used for the reconstruction of functions from general RKHS in $L_2$ where one only has access to the most important eigenfunctions. We verifiably and deterministically construct points and weights for a weighted least squares recovery procedure and pay in the rate of convergence compared to earlier optimal, however probabilistic approaches. The general results apply to the $d$-sphere or multivariate trigonometric polynomials on $\mathbb{T}^d$ spectrally supported on arbitrary finite index sets~$I \subset \mathds{Z}^d$. They can be discretized using at most $|I|^2-|I|+1$ points and weights. Numerical experiments indicate the sharpness of this result. As a negative result we prove that, in general, it is not possible to control the number of points in a reconstructing lattice rule only in the cardinality $|I|$ without additional condition on the structure of $I$. We support our findings with numerical experiments.
    \end{abstract} 

\end{frontmatter}

\section{Introduction} 

In this paper we investigate $D$-optimal designs, which are the result of optimizing the determinant of rescaled Gramian matrices.
This approach is known from optimal experimental designs in statistics literature \cite{DeSt97} or in point constructions for quadrature and interpolation, like Fekete points where one maximizes the determinant of the Vandermonde matrix \cite{KW60,BoPiVi20,Bo22,Bo23}.
We take a frame-theoretical viewpoint and obtain novel results for quadrature and frame analysis.
These two perspectives are connected in the following way:
Let $\bm\varphi(x) = (\varphi_1(x), \dots, \varphi_n(x))^\top\in\mathds C^n$ for $\varphi_1, \dots, \varphi_n \colon\Omega\to\mathds C$ be linearly independent continuous functions on a compact domain $\Omega$.
Then
\begin{itemize}
\item
    $V_n = \Span\{\varphi_1, \dots, \varphi_n\} \subset C(\Omega)$ represents an $n$-dimensional  space of complex-valued functions on a domain $\Omega$ and
\item
    $(\bm\varphi(x))_{x\in\Omega} \subset \mathds C^n$ is a continuous ${\mathds C}^n$-frame indexed by $x\in\Omega$.
\end{itemize}
A $D$-optimal design is a discretization 
with respect to a  probability measure $\varrho^{\ast} = \sum_{i=1}^{M_\varrho}\varrho_i\delta_{x_i}$ which maximizes a certain Gramian determinant, i.e.,
\begin{equation}\label{eq:D}
    \varrho^\ast = \Argmax_{\varrho} \det \Big( \sum_{i=1}^{M_\varrho} \varrho_i \bm\varphi(x_i)\cdot \bm\varphi(x_i)^\ast \Big) \,.
\end{equation}
We utilize this technique in two different variations and obtain twofold results.

\subsection{Parseval subframes and exact discretization} 

For a given (continuous) probability measure $\mu$ on $\Omega$ one may ask for an exact discretization thereof.
In particular, we seek points $x_1, \dots, x_{N}\in\Omega$ and weights $\varrho_1, \dots, \varrho_{N}\ge 0$ (summing up to $1$) forming an exact $L_2$-Marcinkiewicz-Zygmund (MZ) inequality \cite{Marcinkiewicz_Zygmund_1937} for $V_n$ , i.e.,
\begin{equation}\label{exactmz_formula_intro}
    \int_\Omega |f(x)|^2 \;\mathrm d\mu(x) = \sum_{i=1}^{N} \varrho_i |f(x_i)|^2
    \quad\text{for all}\quad
    f\in V_n \,.
\end{equation}
We present an approach using specific $D$-optimal designs to obtain \eqref{exactmz_formula_intro} with at most 
\begin{equation}\label{eq:M}
    N
    = \dim(\Span\{\varphi_k\overline{\varphi_l}\}_{k,l=1}^{n})+1
    \le n^2+1
\end{equation}
points and weights. In addition we give numerical evidence on the sharpness of this result.
More generally, in \Cref{discr_orth}, we present four elementary equivalent conditions under which a system of continuous linearly independent functions $\varphi_1,\dots,\varphi_n$ defined on a compact topological space turns out to be discretely orthonormal with respect to a suitably chosen measure.
We do this by combining an averaging argument with compactness to show that the optimizer of a certain $D$-optimal design procedure \eqref{eq:D} (involving $ \varphi_1, \dots, \varphi_n$) exists and has an objective value of one.
In a second step we use a version of Caratheodory's theorem in order to give the bound \eqref{eq:M} on the required number $N$ of points, see \Cref{CarSub}, which allows for subsampling convex (conic) combinations of (complex) Hermitian matrices.
Caratheodory subsampling has been considered before by several authors, see \cite{BoPiVi20,KPUU24,PiSoVi17}.
The latter reference also comments on its implementation. We further include $L_p$-MZ inequalities for even $p>2$ which supports the findings in \cite{FREEMAN2023126846}. We also contribute to Open Problem 10 in \cite{DaPrTeTi19}, i.e., we give an algorithmic method for constructing an exact $L_p$-MZ inequality with positive weights, cf.\ \Cref{even_MZ}.

Earlier works in this direction rely on the classical 1957 Tschakaloff theorem \cite{Tschak57}, a corner stone for the theory of exact quadrature formulas for spaces of polynomials, see also Shapiro \cite[Thm.\ 3.1.1]{Sha71}, Novak \cite{Nov86}, and Putinar \cite{Pu97} for more general versions. From the results above we even obtain an extended version of the classical Tschakaloff theorem for complex-valued functions, see \Cref{quadrature_formula}. 

In \Cref{Sect:examples}, we apply the general results in two specific contexts. First, we state a general result for the $d$-sphere $\mathds{S}^d$ in \Cref{exactmz_sphere} for even exponents $p\geq 2$. As a consequence we obtain in \Cref{quadrature_sphere}, that there exist points $\bm x^1, \dots, \bm x^N\in\mathds S^d$ forming an exact (weighted) integration rule for the polynomials $\Pi_m$ up to degree $m$ on the $d$-sphere, i.e.,
\begin{equation*}
    \int_{\mathds{S}^d} f(\bm x) \;\mathrm d\mu(\bm x) = \sum_{i=1}^{N} w_i f(\bm x^i) 
    \quad\text{for all}\quad
    f\in \Pi_{m}
\end{equation*}
with
\begin{equation*}
    N \le \dim\Pi_m \le \Big(\frac{9m}{d}\Big)^d \,.
\end{equation*}
Quadrature formulas of the above type with equal weights are called $t$-designs (or $m$-design according to our notation, where $m$ is the degree of the polynomial) introduced in \cite{Delsarte_Goethals_Seidel1977}.

The equal weight condition makes this problem much harder and only a limited number of constructions of spherical designs are known.
However, there are also approaches to obtain spherical designs being exact up to machine precision numerically, cf.\ \cite{Womersley18, GP11}.
In general the existence of spherical designs is known with the optimal asymptotic rate $N\ge C_d m^d$, cf.\ \cite[Thm.\ 1]{BRV13}.
This matches the number of points in our weighted result.

In \Cref{exactmz_torus} we state that \eqref{exactmz_formula_intro} holds for trigonometric polynomials with frequencies in any arbitrary finite index set $I \subset \mathds{Z}^d$.
For
\begin{equation*}
    D(I) \coloneqq \{\bm k - \bm\ell : \bm k, \bm \ell \in I\}
\end{equation*}
the difference set, we find $N \leq |D(I)| \leq |I|^2 - |I|+1$ points and positive weights such that the exact MZ inequality \eqref{exactmz_formula_intro} holds for all $f \in V(I) := \operatorname{span} \{\exp(2\pi i\langle\bm k, \cdot\rangle):\bm k \in I\}$. In general, these points $(x_i)_{i = 1}^N$ should have a very irregular structure. In fact, we show in \Cref{cor:lattice_aliasing} that at least for lattice rules it is not possible to control the number of points only in the cardinality of the index set $|I|$ (without additional conditions on $I$ as in \cite{Kammerer_Potts_Volkmer_2015}). This contradicts \cite[Thm.\ 4.1]{Te17} and contributes to Open Problem 2 in \cite{DaPrTeTi19}.

Also here the question arises whether all of this can be done with equal weights $w_i = 1/N$.
In case of trigonometric polynomials such a discretization is always possible with a large enough  $N$
depending on the ``largest'' frequency in $I$. However, it is still not known, whether the equal-weighted \eqref{exactmz_formula_intro} can be arranged for any index set $I$ with a number of points  $N$
only depending on the cardinality of $|I|$.
We leave this as an open problem.

As our approach is algorithmic, we propose three algorithms in Section~\ref{Sect:discrete_orth}. \Cref{algo0} and \Cref{algo1} are two different ways to determine $N\le n^2+1$ points $(x_i)_i$ and weights $(\varrho_i)_i$ such that \eqref{exactmz_formula_intro} holds true. \Cref{algo2} transforms the basis of $V_n$ to
$\{\psi_1,\dots,\psi_n\}$ and determines $(x_i)_i$ such that the corresponding system of vectors 
$\bm{\psi}(x_i) = (\psi_1(x_i),\ldots,\psi_n(x_i))^\top \in \mathds{C}^n$ constitutes an equal norm tight frame in $\mathds{C}^n$ with respect to some discrete probability measure $\varrho = \sum_{i=1}^{N} \varrho_i\delta_{x_i}$. Several numerical experiments with \Cref{algo1} are conducted in \Cref{sec:numerics} with outcomes aligning with our theoretical findings:
\begin{itemize}
\item
    In Experiment~1 we indicate the numerical applicability of our approach as we recover known lattice point constructions.
    Furthermore, we compute exact MZ points in settings,
    where we know that all lattice rules of that size fail as discussed above.
\item
    In Experiment~2 we investigate a dimensional dependence and observe that for a fixed number of random frequencies fewer points suffice with increasing dimension.
    This suggests a certain ``blessing of dimension'' in this particular setting.
    A similar effect was considered in \cite{GP11} for the sphere $\mathds S^d$.
\item
    Experiment~3 indicates the necessity of the oversampling $N = |D(I)|$ in general as we do not find exact MZ inequalities with fewer points for this particular choice of frequencies in $I$. We do not have a rigorous proof. Theoretical lower bounds indicating at least the quadratic scaling in the number of samples are given in \cite[Thm.\ 3.3, 3.4]{DaPrTeTi19}.   
\end{itemize}

In terms of the language of finite frames \eqref{exactmz_formula_intro} can be rephrased as finding a weighted Parseval subframe $(\sqrt{\varrho_i}\bm\varphi(x_i)))_{i=1}^{N}$ such that
\begin{equation*}
    \|\bm a\|_2^2
    = \int_{\Omega} |\langle \bm a,\bm \varphi(x)\rangle|^2 \;\mathrm d\mu(x)
    = \sum\limits_{i=1}^{N} \varrho_i |\langle \bm a, \bm\varphi(x_i)\rangle|^2
    \quad\text{for all}\quad
    \bm a \in \mathds{C}^n\,.
\end{equation*}
Note, that the problem of the discretization of continuous frames is deeply studied in literature. Here we refer to \cite{FS19}, \cite{FREEMAN2023126846} and the references therein.

\subsection{Equal norm subframes} 

In the well-known paper by Kiefer and Wolfowitz \cite{KW60} the authors prove the equivalence of two different optimization problems. The $D$-optimal design for an $n$-dimensional subspace of real-valued functions  $V_n \subset C(\Omega)$ on the one hand and the min-max problem for the corresponding Christoffel function $\eta_n(x) = \sum_{k=1}^{n}|\varphi_k(x)|^2$ on the other hand. This is the problem of finding a Borel probability measure $\mu$ on $\Omega$ which minimizes $\sup_{f\in V_n \setminus \{0\} } \| f\|^2_{\infty} / \|f\|^2_{L_2(\mu)}$\,.
The latter is sometimes called $G$-optimal design, see Bos~\cite{Bo23}.
The equivalence in \cite{KW60} is useful in a different context.
In the recent paper \cite{KPUU24} two of the authors with coauthors extended the technique in \cite{KW60} to complex-valued functions (where the continuity is even lacking) and the optimal measure turns out to be discrete (Tschakaloff's theorem \cite{Tschak57}, \cite{Pu97} is not required).
Based on the proof in \cite{KPUU24} we prove in \Cref{frame_statement} the existence of a discrete probability measure $\lambda = \sum_{i=1}^{N} \lambda_i \delta_{x_i} $ with $N\leq n^2+1$
and a related linear transformation $\bm{A}_\lambda$ with $\bm{\psi} = \bm{A}_\lambda^{-1/2}\cdot \bm{\varphi}$ such that for a given arbitrary finite (continuous) frame $(\bm{\varphi}(x))_x \subset \mathds{C}^n$ it holds
\begin{equation*}
    \|\bm a\|_2^2 = \sum\limits_{i=1}^{N} \lambda_i |\langle \bm a, \bm\psi(x_i)\rangle|^2 \quad,\quad \bm a \in \mathds{C}^n\,,
\end{equation*}
with the additional feature that $\|\bm{\psi}(x)\|_2 \leq \sqrt{n}$ for all $x\in \Omega$ and $\|\bm{\psi}(x)\|_2 = \sqrt{n}$ on $\operatorname{supp} \lambda$.
This relates to equal norm tight frames (ENTF), which play an important role in the frame community, see \Cref{rem:ENTF}. The corresponding algorithm is formulated in \Cref{algo2}.

\subsection{Guaranteed and verifiable recovery of functions}
We consider a recovery problem, which has been first addressed by Wasilkowski and Wo{\'z}niakowski \cite{WaWo01} in 2001 and drew a lot of attention in the past 6 years \cite{KrUL21,NaSchUl22,DoKrUl23,BSU23}. We aim for the problem to practically find stable recovery algorithms based on suitable sample points and weights which recover verifiably any function from the unit ball of a reproducing kernel Hilbert space (RKHS) $H(k)$ with a prescribed accuracy in $L_2(\mu)$. Here we consider a further restriction which is certainly motivated from practical computation. We only have access to the kernel function and the $n$ most important (largest) singular values and corresponding eigenfunctions of the corresponding integral operator. The main result reads as follows. We construct a linear sampling recovery operator $S^{k,\mu}_{n,N}$ depending on the kernel $k(\cdot, \cdot):\Omega\times \Omega \to \mathds C$ and the target measure $\mu$ which uses $N \leq n^2+1$ many sample points such that   
    \begin{equation}
        \sup\limits_{\|f\|_{H(k)\leq 1}}\|f-S^{k,\mu}_{n,N} f\|_{L_2(\mu)}^2 \leq 3\sum\limits_{j\geq n+1} \sigma_j^2\leq 3\,c_n(H(k),C(\Omega))^2\,,
    \end{equation}
where $c_n(H(k),C(\Omega))$ denotes the $n$-Gelfand number of the embedding $I:H(k) \to C(\Omega)$
\begin{equation*}
    c_n(H(k),C(\Omega))
    = \inf_{\substack{\theta\colon\mathds C^n\to C(\Omega)\\L\in\mathcal L(H(k), \mathds C^n)}}
    \sup_{\|f\|_{H(k)}\le 1}
    \|f-\theta\circ L(f)\|_{C(\Omega)} \,.
\end{equation*}
Our attempt will be to construct an exact discretization for the space of the most important eigenfunctions. In special cases like periodic functions this is done using, e.g., rank-$1$ lattices for the given frequency sets in $\mathds Z^d$, like hyperbolic crosses, in \cite{Kammerer_Potts_Volkmer_2015}. As this specific approach does not work in general, we rely on optimal $D$-designs working in the most general context.

\paragraph{Notation} 

As usual, $\mathds N$, $\mathds Z$,  $\mathds R$, $\mathds C$  denote the natural (without zero), integer, real, and complex numbers.
If not indicated otherwise $\log(\cdot)$ denotes the natural logarithm.
$\mathds C^n$ shall denote the complex $n$-space and $\mathds C^{m\times n}$ the set of complex $m\times n$-matrices.
Vectors and matrices are usually typesetted boldface.
We use  ${\bm y}^\ast:=\overline{\bm y}^\top$.
In general, the adjoint of a matrix $\bm L\in\mathds C^{m\times n}$ is denoted by $\bm L^{\ast}$.
For the spectral norm we use $\|\bm L\|_{2\to 2}$, whereas the Frobenius norm is denoted with $\|\bm L\|_F$. For a compact topological space $\Omega$ we use $C(\Omega)$ for the space of complex-valued continuous functions on $\Omega$. \section{Auxiliary tools}\label{Sect:aux_tools}

\begin{prop}\label{meandeterminant} Let $\psi_1,\dots,\psi_m\colon \Omega\to\mathds C$ be orthonormal with respect to a measure $\mu$ on $\Omega$. Let further $\bm \psi (x) := (\psi_1(x),\dots,\psi_m(x))^\top \in \mathds C^m$.
    Then for fixed $M\in \mathds N$ with $M\geq m$ we have
    \begin{align*}
        &\int_\Omega \dots \int_\Omega
        \det\Big(
        \frac{1}{M}\sum_{i=1}^{M}\bm \psi(x_i)\cdot\bm \psi(x_i)^\ast
        \Big)
        \;\mathrm d\mu(x_1)\dots\;\mathrm d\mu(x_M) \\
        &\quad= \frac{M}{M} \frac{M-1}{M} \cdots \frac{M-m+1}{M} \,.
    \end{align*}
\end{prop} 

\begin{proof} By Leibniz' formula for the determinant we have
    \begin{align}
        \det\Big(
        \frac{1}{M}\sum_{i=1}^{M}\bm \psi(x_i)\cdot\bm \psi(x_i)^\ast
        \Big)
        &= \sum_{\sigma\in S_m} \sgn(\sigma) \prod_{k=1}^{m} \frac{1}{M} \sum_{i=1}^{M} \psi_k(x_i)\overline{\psi_{\sigma(k)}(x_i)} \nonumber\\
        &= M^{-m} \sum_{i_1,\dots,i_m = 1}^{M} \sum_{\sigma\in S_m} \sgn(\sigma)
        \cdot \psi_1(x_{i_1})\overline{\psi_{\sigma(1)}(x_{i_1})} \cdots \psi_m(x_{i_m})\overline{\psi_{\sigma(m)}(x_{i_m})}
        \,. \label{eq:asdljaf}
    \end{align}

    We now show that the inner sum evaluates to zero whenever some $i_l=i_k$ for some $1\le l \neq k \le m$.
    Without loss of generality assume $i_1 = i_2 = j$.
    We split $S_m$ into the permutations with positive and negative sign by using $\sgn(\sigma\circ(1\;2)) = -\sgn(\sigma)$, i.e.,
    \begin{align*}
        S_m
        &= \{\sigma\in S_m : \sgn(\sigma) = 1\} \;\dot\cup\; \{\sigma\in S_m : \sgn(\sigma) = -1\} \\
        &= \{\sigma\in S_m : \sgn(\sigma) = 1\} \;\dot\cup\; \{\sigma\circ(1\;2)\in S_m : \sgn(\sigma) = 1\} \,.
    \end{align*}
    For the inner sum of \eqref{eq:asdljaf} it follows in case $i_1=i_2 = j$
    \begin{align*}
        &\sum_{\sigma\in S_m} \sgn(\sigma)
        \cdot
        \psi_1(x_j)\overline{\psi_{\sigma(1)}(x_j)}
        \cdot
        \psi_2(x_j)\overline{\psi_{\sigma(2)}(x_j)}
        \cdots \psi_m(x_{i_m})\overline{\psi_{\sigma(m)}(x_{i_m})} \\
        &\quad= \sum_{\substack{\sigma\in S_m\\\sgn(\sigma)=1}}
        \psi_1(x_j)\overline{\psi_{\sigma(1)}(x_j)}
        \cdot
        \psi_2(x_j)\overline{\psi_{\sigma(2)}(x_j)}
        \cdots \psi_m(x_{i_m})\overline{\psi_{\sigma(m)}(x_{i_m})} \\
        &\qquad-
        \psi_1(x_j)\overline{\psi_{\sigma(2)}(x_j)}
        \cdot
        \psi_2(x_j)\overline{\psi_{\sigma(1)}(x_j)}
        \cdots \psi_m(x_{i_m})\overline{\psi_{\sigma(m)}(x_{i_m})} \\
        &\quad= 0
        \,.
    \end{align*}
    Thus, in \eqref{eq:asdljaf} we only need to sum over pairwise different indices $i_l \neq i_k$ for $1\le l,k \le m$.

    It remains to apply the integrals. Using Fubini's theorem gives
    \begin{align*}
        &\int_\Omega \dots \int_\Omega
        \det\Big(
        \frac{1}{M}\sum_{i=1}^{M}\bm \psi(x_i)\cdot\bm \psi(x_i)^\ast
        \Big)
        \;\mathrm d\mu(x_1)\dots\;\mathrm d\mu(x_M) \\
        &= M^{-m}
        \sum_{\substack{i_1, \dots, i_m = 1\\i_l\neq i_k}}^{M} \sum_{\sigma\in S_m} \sgn(\sigma)
        \cdot
        \int_\Omega \psi_1(x_{i_1})\overline{\psi_{\sigma(1)}(x_{i_1})} \;\mathrm d\mu(x_{i_1})
        \cdots
        \int_\Omega \psi_m(x_{i_m})\overline{\psi_{\sigma(m)}(x_{i_m})} \;\mathrm d\mu(x_{i_m}) \\
        &= M^{-m}
        \sum_{\substack{i_1, \dots, i_m = 1\\i_l\neq i_k}}^{M} \sum_{\sigma\in S_m} \sgn(\sigma)
        \cdot
        \delta_{1,\sigma(1)}
        \cdots
        \delta_{m,\sigma(m)} \\
        &= M^{-m}
        \sum_{\substack{i_1, \dots, i_m = 1\\i_l\neq i_k}}^{M}
        1
        \,.
    \end{align*}
    The number of elements in $\{i_1,\dots,i_m=1,\dots,M : i_l\neq i_k\}$ evaluates to $M(M-1)\cdots(M-m+1)$, which shows the assertion.
\end{proof} 

\begin{remark} The previous result is a special case of a general result relating the expectation of the determinant of a random matrix and the determinant of the expectation of a matrix, see \cite[Lemma~2.3]{DWD22}. As things simplify in our case we give an elementary proof for the convenience of the reader.
\end{remark} 

\begin{corollary}\label{closetoone} Let $\psi_1,\dots,\psi_m\colon \Omega\to\mathds C$ be orthonormal with respect to a probability measure $\mu$ and $\bm \psi(x) = (\psi_1(x),\dots,\psi_m(x))^\top$.
    Then for all $\varepsilon>0$ there exists an $M\in \mathds N$ and $x_1,\dots,x_M\in \Omega$ such that
    \begin{align*}
        \det\Big(
        \frac{1}{M}\sum_{i=1}^{M}\bm \psi(x_i)\cdot\bm \psi(x_i)^\ast
        \Big)
        \ge 1-\varepsilon \,.
    \end{align*}
\end{corollary} 

\begin{proof} In \Cref{meandeterminant} we computed the mean of the determinant for $m$ orthonormal functions and $M$ points to be $M(M-1)\cdots(M-m+1)/M^m$. Since $\mu$ is assumed to be a probability measure here, there must therefore exist $M$ points $x_1,\dots,x_M\in\Omega$ such that this value is actually attained or exceeded.
    Consequently, as this value tends to $1$ for a growing number of points $M$, the assertion follows.
\end{proof} 

\begin{prop}\label{dettrace} Let $\psi_1,\dots,\psi_m\colon \Omega\to\mathds{C}$ be functions. As above we put
    $\bm \psi(x) := (\psi_1(x),\dots,\psi_m(x))^\top \in \mathds{C}^m$.
    Further, let $M\in \mathds N$, $\alpha_1, \dots, \alpha_M \ge 0$ with $\alpha_1+\dots+\alpha_M = 1$, and $x_1,\dots,x_M\in \Omega$.
    Then
    \begin{equation*}
        \sqrt[m]{\det\Big(
        \sum_{i=1}^{M}\alpha_i \bm \psi(x_i)\cdot\bm \psi(x_i)^\ast
        \Big)}
        \le \frac{1}{m}
        \Tr \Big(\sum_{i=1}^{M}\alpha_i \bm \psi(x_i)\cdot\bm \psi(x_i)^\ast\Big)
        \leq \frac{1}{m}\sup\limits_{x \in \Omega}\sum\limits_{k=1}^m |\psi_k(x)|^2\,.
    \end{equation*}
\end{prop} 

\begin{proof} The first inequality is a consequence of the inequality between arithmetic and geometric mean.
    For a positive semi-definite Hermitian matrix $\bm A\in\mathds C^{m\times m}$ with eigenvalues $\lambda_1(\bm A), \dots, \lambda_m(\bm A)\ge 0$, we have
    \begin{equation*}
        \sqrt[m]{\det(\bm A)}
        = \Big(\prod_{k=1}^{m}\lambda_k(\bm A)\Big)^{1/m}
        \le \frac{1}{m} \sum_{k=1}^{m} \lambda_k(\bm A)
        = \frac{\Tr(\bm A)}{m} \,.
    \end{equation*}
    For the second inequality, we compute the trace of the given matrix via its diagonal entries
    \begin{align*}
        \Tr \Big(\sum_{i=1}^{M}\alpha_i \bm \psi(x_i)\cdot\bm \psi(x_i)^\ast\Big)
        &= \sum_{k=1}^{m}
        \sum_{i=1}^{M}\alpha_i [\bm \psi(x_i)\cdot\bm \psi(x_i)^\ast]_{k,k} \\
        &= \sum_{i=1}^{M}\alpha_i
        \sum_{k=1}^{m} |\psi_k(x_i)|^2 \\
        &\le \Big( \sum_{i=1}^{M}\alpha_i \Big)
        \sup_{x\in \Omega} \sum_{k=1}^{m} |\psi_k(x)|^2 \,.
    \end{align*}
      Taking $\alpha_1 + \dots + \alpha_M =1$ into account, the proof is complete.
\end{proof} 

\begin{lemma}[A version of Lemma 10 from \cite{KPUU24}]\label{Lemma_10_KPUU24}
    Let $\psi_1,\dots,\psi_m\colon \Omega\to\mathds C$ be functions and denote $\bm \psi(x) = (\psi_1(x),\dots,\psi_m(x))^\top$. Then it holds
    \begin{equation*}
        \dim_{\mathds{C}}(\Span_{\mathds{C}}\{ \psi_1,\dots,\psi_m\}) = \dim_{\mathds{C}}(\Span_{\mathds{C}}\{\bm \psi(x) \colon x\in \Omega \})\,,
    \end{equation*}
    where we consider a linear space
    of functions
    on the left-hand side. On the right-hand side a subspace of $\mathds{C}^m$ is considered.
\end{lemma} 

\begin{lemma}\label{lemma_dimensions} Let $\psi_1,\dots,\psi_m\colon \Omega\to\mathds C$ be functions and denote $\bm \psi(x) = (\psi_1(x),\dots,\psi_m(x))^\top$. For the set
    \begin{equation}\label{eq:setH}
        \mathcal{M}:=\{
\bm \psi(x) \cdot \bm \psi(x)^{*}    \colon \ x\in \Omega\}\,
    \end{equation}
    of complex-valued Hermitian matrices in $\mathds C^{m \times m}$ it holds
    \begin{equation*}
        \dim_{\mathds{C}}(\Span_{\mathds{C}}\{\psi_k\overline{\psi_l}\}_{k,l=1}^{m})
   =     \dim_{\mathds{C}}(\Span_{\mathds{C}} \mathcal{M})
        = \dim_{\mathds{R}} (\Span_{\mathds{R}}\mathcal{M})
        \le m^2 \,.
    \end{equation*}
\end{lemma} 

\begin{proof} We consider the set \eqref{eq:setH}
of matrices with the entries  $\psi_k(x)\overline{\psi_l(x)}$, $k,l=1,\dots, m$.
    These matrices can also be interpreted as vectors $\tilde{\bm \psi}(x) :=(f_1(x), \dots, f_{m^2}(x))^\top \in \mathds{C}^{m^2}$ with component functions $f_j:\Omega\to\mathds{C}$, $j=1,\dots, m^2$. Further, $\tilde{\mathcal{M}} := \{
    \tilde{\bm \psi}(x) \colon \ x\in \Omega\} $ can be identified with $\mathcal{M}$.
    With \Cref{Lemma_10_KPUU24} we then obtain
    \begin{equation*}
        \dim_{\mathds{C}}(\Span_{\mathds{C}}\{\psi_k\overline{\psi_l}\}_{k,l=1}^{m}) =
        \dim_{\mathds{C}}(\Span_{\mathds{C}}\{f_j\}_{j=1}^{m^2}) = \dim_{\mathds{C}}(\Span_{\mathds{C}}\tilde{\mathcal{M}})
        =\dim_{\mathds{C}}(\Span_{\mathds{C}}\mathcal{M}) \,,
    \end{equation*}
    where $\dim_{\mathds{C}}(\Span_{\mathds{C}}\{f_j\}_{j=1}^{m^2}) \le m^2$ is obvious. Further,
    the space $\Span_{\mathds{C}}\mathcal{M}$ can be decomposed as a direct sum of $\Span_{\mathds{R}}\mathcal{M}$ and $\mathrm{i}\cdot \Span_{\mathds{R}}\mathcal{M}$, i.e.,
    \begin{equation*}
        \Span_{\mathds{C}}\mathcal{M} =  \Span_{\mathds{R}}\mathcal{M}  \oplus_{\mathds{R}} \mathrm{i}\cdot \Span_{\mathds{R}}\mathcal{M} \,.
    \end{equation*}
    For this, note that
    $\Span_{\mathds{R}}\mathcal{M}$ is an $\mathds{R}$-linear space of Hermitian matrices,
    whereas $\mathrm{i}\cdot \Span_{\mathds{R}}\mathcal{M}$ is an $\mathds{R}$-linear space consisting of skew-Hermitian matrices. Therefore $\Span_{\mathds{R}}\mathcal{M}\cap (\mathrm{i}\cdot \Span_{\mathds{R}}\mathcal{M}) =\{ 0 \}$, and for the dimensions we deduce $\dim_{\mathds{R}}(\Span_{\mathds{R}} \mathcal{M}) = \dim_{\mathds{R}} (\mathrm{i}\cdot \Span_{\mathds{R}}\mathcal{M})$ and
    \begin{equation*}
        \dim_{\mathds{C}}(\Span_{\mathds{C}} \mathcal{M}) = \frac{1}{2} \dim_{\mathds{R}}(\Span_{\mathds{C}}\mathcal{M})
        = \frac{1}{2} \dim_{\mathds{R}} (\Span_{\mathds{R}}\mathcal{M}\oplus_{\mathds{R}} \mathrm{i}\cdot \Span_{\mathds{R}}\mathcal{M}) =  \dim_{\mathds{R}} (\Span_{\mathds{R}}\mathcal{M}) \,.
    \end{equation*}
    The proof is finished.
\end{proof} 

\begin{remark} Note that in general for complex-valued functions  $\psi_1,\dots,\psi_m\colon \Omega\to\mathds C$, in contrast to the set $\mathcal{M}$ of Hermitian matrices,  we have
    \begin{equation*}
        \dim_{\mathds{C}}(\Span_{\mathds{C}}\{\psi_k\overline{\psi_l}\}_{k,l=1}^{m}) \neq  \dim_{\mathds{R}}(\Span_{\mathds{R}}\{\psi_k\overline{\psi_l}\}_{k,l=1}^{m}) .
    \end{equation*}
    The poof of \Cref{lemma_dimensions} heavily depends on the structure of sets in the underlying spans. A simple example is the system $\psi_1 = 1$, $\psi_2 = \mathrm i$, for which one gets
    \begin{equation*}
        \dim_{\mathds{C}}(\Span_{\mathds{C}}\{\psi_k\overline{\psi_l}\}_{k,l=1}^{2}) = 1 \neq 2 =
        \dim_{\mathds{R}}(\Span_{\mathds{R}}\{\psi_k\overline{\psi_l}\}_{k,l=1}^{2}) .
    \end{equation*}
\end{remark} 

\begin{prop}[Caratheodory subsampling for Hermitian matrices]\label{CarSub} Let $\psi_1,\dots,\psi_m\colon \Omega\to\mathds C$ be  functions and denote $\bm \psi(x) = (\psi_1(x),\dots,\psi_m(x))^\top$.
    Further, let $M\in \mathds N$, $\alpha_1, \dots, \alpha_M \ge 0$ and $x_1,\dots,x_M\in \Omega$ such that
    \begin{equation*}
        \bm A = \sum_{i=1}^{M}\alpha_i \bm \psi(x_i)\cdot\bm \psi(x_i)^\ast \,.
    \end{equation*}
    Let further $N = \dim(\Span\{\psi_k\overline{\psi_l}\}_{k,l=1}^{m}) \le m^2$. Then the following holds.

    \begin{enumerate}[(i)]
    \item
        There exists a subset $\{y_1,\dots,y_ {N} \} \subset \{x_1,\dots,x_M\} \subset \Omega$ of the initial nodes and weights $\beta_1,\dots,\beta_N\ge 0$ such that $\bm A$ can be represented as
        \begin{equation*}
            \bm A = \sum_{i=1}^N\beta_i \bm \psi(y_i)\cdot\bm \psi(y_i)^\ast \,.
        \end{equation*}
    
    \item
        If additionally $\alpha_1+\dots+\alpha_M = 1$ then we have a subset
        $\{z_1,\dots,z_{N+1}\} \subset \{x_1,\dots,x_M\} \subset \Omega$ and weights $\gamma_1,\dots,\gamma_{N+1}\geq 0$ satisfying $\gamma_1+\dots+\gamma_{N+1}= 1$ with
        \begin{equation*}
            \bm A = \sum_{i=1}^{N+1}\gamma_i \bm \psi(z_i)\cdot\bm \psi(z_i)^\ast \,.
        \end{equation*}
    \end{enumerate}
\end{prop}

\begin{proof} The convex hull of the set \eqref{eq:setH} of matrices $\mathcal M$ is defined via
    \begin{equation*}
       \conv\mathcal M := \bigg\{\sum_{i=1}^M \alpha_i
\bm \psi(x_i) \cdot \bm \psi(x_i)^{*}
       \colon\, M\in \mathds{N},\, \alpha_i \ge 0,\,\sum_{i=1}^M \alpha_i = 1,\, x_i\in \Omega\bigg\}
    \end{equation*}
    and clearly lies in $\Span_{\mathds{R}}\mathcal{M}$. Consequently, it can be considered a convex subset of $\mathds{R}^N$ (see \Cref{lemma_dimensions}).
    By Caratheodory's theorem, see \cite[Thm.\ 1.1.4]{Schn14},
    we have that each element of a convex hull, which is a subset of $\mathds{R}^N$, can be represented as a convex combination
    of $N+1$ elements. By the above identification this reduction transfers to $\conv\mathcal M$ and proves (ii).

    Considering instead of $\conv\mathcal M$ the conic hull and, respectively, dropping the condition that the coefficients $\alpha_i$ sum up to one, a similar reduction can be shown with a straight forward modification of the proof. Without the convexity assumption we may even reduce to $N$ summands (instead of $N+1$). This gives (i).

\end{proof} %
 \section{Discrete orthogonality and tight frames via $D$-optimal designs}\label{Sect:discrete_orth}

We are now heading for conditions under which a given system of linearly independent, bounded and continuous functions
$\varphi_1,\dots,\varphi_n\colon \Omega\to\mathds{C}$ turns out to be (discretely) orthonormal with respect to some probability measure $\mu$ on $\Omega$. We will prove a characterization first, which is interesting on its own.

In case the characteristic function $\mathds 1_\Omega$ belongs to $V_n:=\operatorname{span}_{\mathds{C}}\{\varphi_1(\cdot),\dots,\varphi_n(\cdot)\}$ we may define
\begin{equation*}
    \eta_n(x):= \frac{\sum_{k=1}^n |\varphi_k(x)|^2}{n},
\end{equation*}
which is a continuous function on $\Omega$ bounded away from zero (otherwise the $\mathds 1_\Omega$-function would have zeros). In case the $\mathds 1_\Omega$-function is not contained in $V_n$ we blow up the space to
\begin{equation*}
    V_{n+1}
    := \operatorname{span}_{\mathds C}\{\varphi_1(\cdot),\cdots,\varphi_n(\cdot), \varphi_{n+1} := \mathds 1_\Omega\}\,.
\end{equation*}
In this case we obtain a basis $\{\varphi_1(\cdot),\cdots,\varphi_n(\cdot),\mathds 1_\Omega\}$ and the corresponding Christoffel-type function
\begin{equation*}
    \eta_{n+1}(x)
    \coloneqq \frac{1+\sum_{k=1}^n |\varphi_k(x)|^2}{n+1},
\end{equation*}
which is again bounded away from zero and continuous at any $x\in \Omega$. 

Next, consider a discrete (atomic) measure on $\Omega$
\begin{equation}\label{eqdef:discrete_measure}
\varrho := \sum_{i=1}^{M_{\varrho}} \varrho_i \delta_{x_i} \,,
\end{equation}
where the non-negative weights $(\varrho_i)_i$ sum up to one and the points $(x_i)_i \subset \Omega$ are fixed.
Further, define
\begin{equation*}
    \langle f,g\rangle_{\varrho}
    \coloneqq \sum\limits_{i=1}^{M_\varrho} \varrho_i f(x_i)\overline{g(x_i)}
\end{equation*}
and the $m \times m$ Gramian
\begin{equation}\label{Arho}
    \bm{A}_\varrho
    := \Big(\Big\langle \frac{\varphi_k}{\sqrt{\eta_m}},\frac{\varphi_l}{\sqrt{\eta_m}} \Big\rangle_{\varrho}\Big)_{k,l = 1}^{m}\,,
\end{equation}
where either $m:=n$ or $m:=n+1$ depending on the cases above.
The matrix $\bm{A}_{\varrho}$ can then also be represented in the form $\bm{A}_{\varrho} = \sum_{i=1}^{M_\varrho} \varrho_i \bm{\psi}(x_i)\cdot \bm{\psi}(x_i)^*$ with $\bm{\psi}(x):= (\psi_1(x), \dots, \psi_m(x))^\top$
and
\begin{equation}\label{rescaling_of_varphi}
    \psi_k(x) := \varphi_k(x)/\sqrt{\eta_m(x)} \quad,\quad k=1,\dots,m \,.
\end{equation}
This transition from the functions $\varphi_k$ to the functions $\psi_k$ performs the ``change of measure''.

Let us finally consider the corresponding set $\mathcal{M}$ from \eqref{eq:setH} in the proof of \Cref{CarSub}. This is clearly a compact set of rank-1 matrices
$\bm{\psi}(x) \cdot \bm{\psi}(x)^{*} \in \mathds C^{m \times m}$ if $\Omega$ is assumed to be compact due to the continuity
of the mapping $x \mapsto \bm \psi(x)\cdot {\bm \psi}(x)^*$. Besides, each element of $\mathcal{M}$ is  Hermitian and positive semi-definite with real determinant, i.e.,
$\det (\bm{\psi}(x) \cdot \bm{\psi}(x)^{*}) \geq 0$ for all $x\in \Omega$.
Furthermore, by \Cref{CarSub} the convex hull remains closed and bounded and hence compact.

The determinant is a continuous mapping on $\mathds{C}^{m\times m}$, which implies the existence of a maximizer $\bm A_{\alpha} = \sum_{i=1}^{M_{\alpha}} \alpha_i \bm \psi(x_i)\cdot\bm \psi(x_i)^* \in \conv\mathcal{M}$ which satisfies (for weights $\varrho$ as in~\eqref{eqdef:discrete_measure} and $\bm{A}_\varrho$ as in~\eqref{Arho})
\begin{equation}\label{Aalpha}
  \det(\bm A_{\alpha}) = \max\limits_{\varrho} \det(\bm{A}_{\varrho}) = \sup\limits_{\bm M \in \conv \mathcal{M}} \det(\bm M)\,.
\end{equation}
From \Cref{dettrace} we obtain that $\det(\bm{A}_{\alpha}) \leq 1$ since $\sum_{k=1}^m |\psi_k(x)|^2 = m$\,.

The subsequent proposition characterizes the property that a set of given continuous functions turns out to be orthonormal with respect to a suitably chosen Borel probability measure $\mu$. This result might be of independent interest. Note, that in case the constant function belongs to their span, i.e., $\sum_{i=1}^n \gamma_i\varphi_i = \mathds 1_\Omega$, the condition $\sum_{i=1}^n |\gamma_i|^2 = 1$ is required. However, this can be guaranteed by a simple rescaling of the basis functions.      

\begin{prop}\label{discr_orth} Let $\Omega$ be a compact topological space and $\varphi_1,\dots,\varphi_n\colon \Omega\to\mathds C$  linearly independent continuous functions on $\Omega$. Let further $N = \dim(\Span\{\varphi_k\overline{\varphi_l}\}_{k,l=1}^{n})$. The following assertions are equivalent.
    \begin{enumerate}[(i)]
    \item
        We have
        \begin{equation*}
            {\bm I}_n \in \conv \bigg\{
            (\varphi_k(x)\overline{\varphi_l(x)}
            )_{k,l=1}^{n}
            \colon\,x\in \Omega \bigg\}\,.
        \end{equation*}

    \item
        There exist $N+1$ points $x_1,\dots,x_{N+1}\in \Omega$  and non-negative weights with $\mu_1+\dots+\mu_{N+1} = 1$ such that
        \begin{enumerate}
        \item[\em (ii\,a)]
            the system $(\varphi_l)_{l=1}^{n}$ is orthonormal with respect to discrete measure $\varrho = \sum_{i=1}^{N+1} \mu_i \delta_{x_i}$\,, i.e.,
\begin{equation*}
                \sum\limits_{i=1}^{N+1} \mu_i\varphi_k(x_i)\overline{\varphi_l(x_i)}
                = \delta_{k,l}\,, \quad  k,l=1,\dots, n;
            \end{equation*}
\item[\em (ii\,b)]
            we have
            \begin{equation}\label{eq1000}
                \min\limits_{(x_i)_i,(\mu_i)_i} \Big\|\sum_{i=1}^{N+1} \mu_i\bm{\varphi}(x_i)\cdot \bm{\varphi}(x_i)^*-{\bm I}_n\Big\|_F = 0
            \end{equation}
            for the minimization problem with respect to the Frobenius norm $\|\cdot\|_F$\,.\\ Here $\bm{\varphi}(x) = (\varphi_1(x),\dots,\varphi_n(x))^T$\,.
        \end{enumerate}

    \item
        There exists a Borel probability measure $\mu$ on $\Omega$ such that
        \begin{equation*}
            \int_\Omega \varphi_k(x)\overline{\varphi_l(x)} \;\mathrm d\mu(x)
            = \delta_{k,l}\,, \quad k,l=1,\dots, n,
        \end{equation*}
        i.e., the system
$(\varphi_l)_{l=1}^{n}$
        is orthonormal with respect to the measure $\mu$.

    \item The representation $\mathds 1_\Omega=\sum_{i=1}^{m} \gamma_i\varphi_i$ holds true with $\sum_{i=1}^m |\gamma_i|^2 = 1$ (Recall that $\varphi_{n+1}:=\mathds 1_\Omega$ in case $m=n+1$.) and the maximizer $\bm{A}_\alpha$ in~\eqref{Aalpha} satisfies
    one of the following two (equivalent) conditions:
        \begin{enumerate}
        \item[\em (iv\,a)]
        \begin{equation*}\det(\bm{A}_{\alpha})
= 1 \,,
        \end{equation*}
        \item[\em (iv\,b)]
        \begin{equation*}
              \bm{A}_{\alpha} = \bm{I}_m\,.
        \end{equation*}
        \end{enumerate}    
    \end{enumerate}
\end{prop} 

\begin{proof} {\em Step~1.} Using \Cref{CarSub} yields (i) $\Longrightarrow$ (ii). The implications (ii\,a) $\Longleftrightarrow$ (ii\,b) $\Longrightarrow$ (iii) are immediate.

    {\em Step~2.} We prove (iii) $\Longrightarrow$ (iv). First, note that in case $m=n$ indeed
    \begin{equation*}
    1 = \int_\Omega \mathds 1_\Omega(x) \;\mathrm d\mu(x) = \sum_{k,l=1}^{m} \gamma_k \overline{\gamma_l} \int_\Omega \varphi_k \overline{\varphi_l(x)} \;\mathrm d\mu(x)
            =  \sum_{k,l=1}^{m} \gamma_k \overline{\gamma_l}\delta_{k,l} = \sum_{k=1}^{m} |\gamma_k|^2  \,.
    \end{equation*}
    In case $m=n+1$ we have $\sum_{i=1}^m |\gamma_i|^2 = 1$ due to $\gamma_{m}=1$ and $\gamma_{1}=\dots=\gamma_n=0$.
    Second, we already know that the maximizer $\bm{A}_\alpha$ in \eqref{Aalpha} exists and satisfies $\det(\bm{A}_{\alpha}) \leq 1$. From \Cref{closetoone} applied to the functions $\psi_k$, $k=1,\dots,m$, with the corresponding measure $\eta_m(\cdot)\mathrm d\mu(\cdot)$ we obtain that $\det(\bm A_\alpha) \geq 1-\varepsilon$ for all $\varepsilon>0$ and hence $\det(\bm A_{\alpha}) \geq 1$, which proves (iv\,a).
    Further, from \Cref{dettrace}, it follows that
    $\frac{1}{m} \sum_{k=1}^{m} \lambda_k(\bm A_\alpha) =1$ due to $\sum_{k=1}^m |\psi_k(x)|^2 = m$.
    Therefore, the geometric mean of the eigenvalues of $\bm{A}_\varrho$ equals the arithmetic mean (see \Cref{dettrace}) which is only possible for equal eigenvalues. This implies $\lambda_1 = \cdots = \lambda_{m}=1$ and hence $\bm{A}_\alpha = \bm{I}_{m}$ since $\bm{I}_{m}$ is the only Hermitian matrix that has all eigenvalues 1, i.e., the statement (iv\,b).
    
    {\em Step~3.} It remains to prove the implication (iv) $\Longrightarrow$ (i).
    First note that if (iv\,a) is fulfilled for a maximizer $\bm{A}_{\alpha}$ also (iv\,b) is fullfilled by the arguments in Step~2. Now we distinguish the cases $m=n+1$ and $m=n$.
    
    In case $m=n+1$  we obtain from $\varphi_{n+1} \equiv 1$,
    \begin{equation}\label{convexsum}
        1
        = \Big\langle \frac{\varphi_{m}}{\sqrt{\eta_m}},\frac{\varphi_{m}}{\sqrt{\eta_m}}\Big\rangle_{\alpha}
        = \sum\limits_{i=1}^{M_\alpha} \frac{\alpha_i}{\eta_m(x_i)}\,.
    \end{equation}
    In addition, we obtain for $1\leq k,l \leq n$
    \begin{equation}\label{orthogonality_relations}
        \delta_{k,l}
        = \Big\langle \frac{\varphi_k}{\sqrt{\eta_m}}, \frac{\varphi_l}{\sqrt{\eta_m}} \Big\rangle_{\alpha}
        = \sum\limits_{i=1}^{M_\alpha} \frac{\alpha_i}{\eta_m(x_i)}\varphi_k(x_i)\overline{\varphi_l(x_i)}\,.
    \end{equation}
    Together with \eqref{convexsum} this gives (i). The case $m=n$ is a bit more involved, since in this case we may not have $\varphi_k \equiv 1$ for any $k=1,\dots,n$.
    Clearly, \eqref{orthogonality_relations} is still satisfied.
    To obtain an analogon to~\eqref{convexsum}, we now exploit the condition $\mathds 1_\Omega=\sum_{i=1}^{n} \gamma_i\varphi_i$ with $\sum_{i=1}^n |\gamma_i|^2 =1$. Indeed, then
    \begin{equation*}
        1 = \sum_{k,l=1}^{n} \gamma_k \overline{\gamma_l} \delta_{k,l}
        = \sum_{k,l=1}^{n} \sum\limits_{i=1}^{M_\alpha} \gamma_k \overline{\gamma_l} \frac{\alpha_i}{\eta_n(x_i)}\varphi_k(x_i)\overline{\varphi_l(x_i)}
        = \sum\limits_{i=1}^{M_\alpha} \frac{\alpha_i}{\eta_n(x_i)}\,.
    \end{equation*}
    The proof is finished.
\end{proof} 

\noindent\begin{algorithm}[t] \caption{\texttt{Construction of a discrete Parseval frame}}\label{algo0}
    \begin{tabularx}{\textwidth}{lp{100pt}X}
        \textbf{Input:} &
        $\Omega$ &
        compact topological space and Borel probability measure $\mu$ \\
        &
        $\varphi_1, \dots, \varphi_n\colon \Omega\to\mathds C$ &
        mutually orthonormal and continuous functions on $\Omega$ w.r.t. $\mu$ \\[2pt]
        \hline\\[-8pt]
        \textbf{Output:} &
        $(x_i)_i$ and $(\mu_i)_i$ &
        satisfying \eqref{f100} and $\mu_i$ summing up to $1$, $i=1,\dots,N+1$
        \\[2pt]\hline
    \end{tabularx}
    \begin{algorithmic}[1]
        \STATE{ put $N \coloneqq \dim(\Span\{\varphi_k\overline\varphi_l\}_{k,l=1}^{n})$ }
        \IF{$\mathds 1_\Omega\in\Span_{\mathds C}\{\varphi_1, \dots, \varphi_n\}$}
        \STATE{ put $m\coloneqq n$ }
        \ELSE
        \STATE{ put $\varphi_{n+1} \coloneqq  \mathds 1_\Omega$ and $m\coloneqq n+1$ }
        \ENDIF
        \STATE{ put $\tilde{N} \coloneqq \dim(\Span\{\varphi_k\overline\varphi_l\}_{k,l=1}^{m})$ }
        \STATE{ put $\eta_m(x) \coloneqq  \frac{1}{m}\sum_{k=1}^{m}|\varphi_k(x)|^2$ and $\bm{\varphi}(x) \coloneqq (\varphi_1(x), \dots, \varphi_m(x))^\top$ }
        \STATE{ put $\mathcal M \coloneqq  \Big\{\frac{\bm{\varphi}(x)\cdot \bm{\varphi}(x)^\ast}{\eta_m(x)} : x\in \Omega\Big\}$ }
\STATE{
            find $x_1, \dots, x_{\tilde{N}+1}\in \Omega$ and $\alpha_1, \dots, \alpha_{\tilde{N}+1}\ge 0$ with $\alpha_1+\dots+\alpha_{\tilde{N}+1}=1$ such that
            \begin{equation*}
                \det\Big(\sum_{i=1}^{\tilde{N}+1}\alpha_i \frac{\bm{\varphi}(x_i)\cdot \bm{\varphi}(x_i)^\ast}{\eta_m(x_i)}\Big)
                = \max_{\bm M\in\conv\mathcal M} \det(\bm M)
                = 1
            \end{equation*}
            }
        \IF{$N=\tilde{N}$}
        \STATE{ put  $\mu_i \coloneqq  \alpha_i/\eta_m(x_i)$ for $i=1, \dots, N+1$ }
        \ELSE
        \STATE{ use Caratheodory subsampling to reduce the number of points to $N+1$ and calculate corresponding weights $\mu_i$ summing up to $1$}
        \ENDIF
\RETURN{$(x_i)_i$ and $(\mu_i)_i$}
    \end{algorithmic}
\end{algorithm}

\noindent\begin{algorithm}[t] \caption{\texttt{Construction of a discrete Parseval frame}}\label{algo1}
    \begin{tabularx}{\textwidth}{lp{100pt}X}
        \textbf{Input:} &
        $\Omega$ &
        compact topological space and Borel probability measure $\mu$ \\
        &
        $\varphi_1, \dots, \varphi_n\colon \Omega\to\mathds C$ &
        mutually orthonormal and continuous functions on $\Omega$ w.r.t. $\mu$ \\[2pt]
        \hline\\[-8pt]
        \textbf{Output:} &
        $(x_i)_i$ and $(\mu_i)_i$ &
        satisfying \eqref{f100} and $\mu_i$ summing up to $1$, $i=1,\dots,N+1$
        \\[2pt]\hline
    \end{tabularx}
    \begin{algorithmic}[1]
         \STATE{
$\bm{\varphi}(x) \coloneqq (\varphi_1(x), \dots, \varphi_n(x))^\top$ }
        \STATE{ put $N \coloneqq \dim(\Span\{\varphi_k\overline\varphi_l\}_{k,l=1}^{n})$ }
        \STATE{
            find $x_1, \dots, x_{N+1}\in \Omega$ and $\mu_1, \dots, \mu_{N+1}\ge 0$ with $\mu_1+\dots+\mu_{N+1}=1$ such that
            \begin{equation*}
                \Big\|\sum_{i=1}^{N+1}\mu_i \bm{\varphi}(x_i) \cdot \bm{\varphi}(x_i)^\ast - \bm I_n\Big\|_{F}^2
                = 0
            \end{equation*}
              }
        \RETURN{$(x_i)_i$ and $(\mu_i)_i$}
    \end{algorithmic}
\end{algorithm} 

\begin{remark}\label{rem:finite} If $\Omega$ is a finite set and $\varphi_1,\dots,\varphi_n:\Omega\to \mathds{C}$ linearly independent vectors in $\mathds{C}^{|\Omega|}$ then the characterization in \Cref{discr_orth} remains valid.
\end{remark} 

\begin{remark} \begin{enumerate}[(i)]
    \item
    In the language of the frame community the result in \Cref{discr_orth} essentially tells us that if we start with a continuous Parseval frame $(\bm{\varphi}(x))_{x\in \Omega}$ in the sense that ($\mu$ Borel probability measure, see also Definition 1.2 in \cite{FS19})
    \begin{equation}\label{tight_frame}
        \|\bm{a}\|_2^2
        = \int_\Omega |\langle \bm{a}, \bm{\varphi}(x)\rangle|^2\;\mathrm d\mu(x)\quad,\quad \bm{a} \in \mathds{C}^n\,,
    \end{equation}
    then there are non-negative weights $\mu_i$, which sum up to one, and points $x_i$ such that
    \begin{equation}\label{f100}
        \|\bm{a}\|_2^2 = \sum\limits_{i=1}^{N+1}\mu_i|\langle \bm{a}, \bm{\varphi}(x_i)\rangle|^2\quad,\quad \bm{a} \in \mathds{C}^n\,.
    \end{equation}
    In other words, the frame is still Parseval with respect to a discrete probability measure with at most $N+1$ atoms. We do not just prove their existence, we rather give concrete methods how to obtain these points and weights, cf.\ \Cref{algo0} and \Cref{algo1}.

    \item
    Note, that in the classical frame setting the system of vectors
    \begin{equation}\label{disc_frame}
        \bm{\varphi}_\mu(x_i) = \sqrt{\mu_i}(\varphi_1(x_i),\dots,\varphi_n(x_i))^\top \in \mathds{C}^n\quad,\quad i=1,\dots,N+1,
    \end{equation}
    constructed in \Cref{discr_orth}, constitutes a discrete Parseval frame in $\mathds{C}^n$. Hence, there is a relation to the notion of ``scalable frames'' as introduced in \cite{Cahill_Chen_2013, Kutyniok_Okoudjou_Philipp,Casazza_Chen_2017}. Let us emphasize that the property in \eqref{f100} is stronger, since we require
    $\sum_{i=1}^{N+1} \mu_i = 1$.
    \end{enumerate}
\end{remark} 

\begin{remark}\label{rem:constant_in_Vn} If we have $\mathds 1_\Omega \in V_n$ from the start (and therefore $m=n$ in \eqref{Arho}), the condition~(i) in \Cref{discr_orth}, namely
    ${\bm I}_n \in \conv \{ (\varphi_k(x)\overline{\varphi_l(x)} )_{k,l=1}^{n} \colon\,x\in \Omega \}$,
    is equivalent to
   \begin{flalign*}
        \mathrm{(i')}
        &&
        {\bm I}_n \in \operatorname{cone} \bigg\{
        (\varphi_k(x)\overline{\varphi_l(x)}
        )_{k,l=1}^{n}
        \colon\,x\in \Omega \bigg\}\,,
        &&
   \end{flalign*}
   i.e., we need not require the summability to 1 of the coefficients explicitly.
    In this case, the number of points needed in (ii) can be reduced to $N$.
\end{remark}

\subsection{Equal norm tight frames with respect to a probability measure} \label{sec:scalable_frames}

Assume that $\varphi_1,\ldots,\varphi_n:\Omega\to\mathds{C}$ are orthonormal functions on a domain $\Omega$ with respect to some probability measure $\mu$ on $\Omega$
or, in other words, that the corresponding system of vectors 
$\bm{\varphi}(x) = (\varphi_1(x),\ldots,\varphi_n(x))^\top \in \mathds{C}^n$ represents a Parseval frame in $\mathds{C}^{n}$. In this setting condition~(iii) of \Cref{discr_orth} is fulfilled and from the previous consideration we learned that, as long as the set $\mathcal{M}$ in \eqref{eq:setH} is a compact subset of $\mathds{C}^{n\times n}$, we always find a discrete probability measure $\varrho = \sum_{i=1}^{M_{\varrho}} \varrho_i\delta_{x_i}$ via special $D$-optimal designs preserving the orthogonality. 
These are essentially the statements of \Cref{discr_orth} and \Cref{rem:finite}. 

We now again start with a system of functions $\varphi_1,\ldots,\varphi_n:\Omega \to \mathds{C}$ and, in what follows, we again require the compactness of the set $\mathcal{M}$ from \eqref{eq:setH}. This would be the case, for example, in the considered settings of \Cref{discr_orth} and \Cref{rem:finite}, where either $\Omega$ is finite or $\Omega$ is a compact topological space and $\varphi_i:\Omega \to \mathds{C}$, $i=1,\ldots,n$, represent continuous functions. However, this time we do not demand orthonormality of the functions and thus do not need a measure $\mu$ on $\Omega$ in the beginning. Instead, we are satisfied with linear independence. The functions then define an $n$-dimensional linear function space $V_n$.

Interestingly, it turned out in \cite{KW60, KPUU24} that $D$-optimal designs are actually useful to additionally control the norms $\|f\|_{\infty}$ of the functions $f\in V_n$. In fact, as we will see below, we always find a discrete probability measure $\lambda= \sum_{i=1}^{N} \lambda_i\delta_{x_i}$ on $\Omega$ with $N\le n^2+1$ such that 
\begin{equation}\label{Sup_Estimate}
\sup\limits_{x\in\Omega} |f(x)| \leq \sqrt{n} \Big(\sum_{i=1}^{N} \lambda_i |f(x_i)|^2\Big)^{1/2}
\quad\text{for all }f\in V_n \,.
\end{equation}
As above, denote with $\varrho = \sum_{i=1}^{M_{\varrho}} \varrho_i\delta_{x_i}$ a discrete probability measure on $\Omega$. Putting $\bm{A}_{\varrho} = \sum_{i=1}^{M_{\varrho}} \varrho_i\bm{\varphi}(x_i)\cdot \bm{\varphi}(x_i)^*$, due to the compactness of  $\mathcal{M}$, 
we then also have a maximizer
\begin{equation}\label{Aalpha2}
    \det(\bm A_{\lambda}) = \max\limits_{\varrho} \det(\bm{A}_{\varrho}) = \sup\limits_{\bm M \in \conv \mathcal{M}} \det(\bm M)\,.
\end{equation}
This time a rescaling by the Christoffel function as in \eqref{rescaling_of_varphi} is not necessary.
Taking Caratheodory's theorem into account, the measure $\lambda$ gives non-negative weights $(\lambda_i)_{i=1}^{N}$ and points $(x_i)_{i=1}^{N}$ with  $N =\dim(\Span\{\varphi_k\overline\varphi_l\}_{k,l=1}^{n})+1$.

The following theorem states our result in the case of a compact topological space $\Omega$ and continuous functions $\varphi_i:\Omega \to \mathds{C}$, $i=1,\ldots,n$, i.e., in the setting corresponding to \Cref{discr_orth}. An essential part of it is actually proved in \cite{KPUU24}.

\begin{theorem}\label{frame_statement} Let $\Omega$ denote a compact topological space and $\bm{\varphi}(x) = (\varphi_1(x),\dots,\varphi_n(x))^\top$ a vector of $n$ linearly independent continuous functions on $\Omega$. Further, let $N=\dim(\Span\{\varphi_k\overline\varphi_l\}_{k,l=1}^{n})$. \Cref{algo2} then results in a discrete probability measure $\lambda = \sum_{i=1}^{N+1} \lambda_i \delta_{x_i}$ and a corresponding strictly positive definite Hermitian matrix $\bm{A}_{\lambda} = \sum_{i=1}^{N+1} \lambda_i \bm{\varphi}(x_i) \cdot\bm{\varphi}(x_i)^*$ such that the new system  $\bm{\psi}(x) := \mathbf{A}_{\lambda}^{-1/2}\cdot \bm{\varphi}(x)$ satisfies
    \begin{equation}\label{eq101}
        \sum\limits_{i=1}^{N+1} \lambda_i|\langle \bm{a}, \bm{\psi}(x_i) \rangle|^2  = \|\bm{a}\|_2^2 \,
    \end{equation}
    for all $\bm{a} \in \mathds{C}^n$. In addition, we have $\|\bm{\psi}(x)\|_2 \leq \sqrt{n}$ for any $x\in \Omega$ and $\|\bm{\psi}(x)\|_2 = \sqrt{n}$ on $\operatorname{supp} \lambda$.
\end{theorem} \begin{proof} The orthogonality relation in \eqref{eq101} with respect to the induced inner product coming from the discrete measure $\lambda$  is straightforward and a consequence of $\bm{\psi} = \bm{A}_{\lambda}^{-1/2}\cdot \bm{\varphi}$. At this point it is not relevant that $\lambda$ is the maximizing measure in \eqref{Aalpha2}. In fact, note that for any full rank matrix ${\bm B} \in \mathds{C}^{n\times(N+1)}$ and ${\bm A}:= {\bm B}\cdot {\bm B}^{\ast}$ one has
    \begin{equation*}
        ({\bm A}^{-1/2}\cdot{\bm B}) \cdot({\bm A}^{-1/2}\cdot{\bm B})^{\ast} = {\bm I}_n \,.
    \end{equation*}
    Applying this identity to $\bm B = (\sqrt{\lambda_j}\varphi_i(x_j))_{i=1,j=1}^{n,N+1}$ yields \eqref{eq101}.

    The property that $\lambda$ is the maximizing measure becomes relevant to prove $\|\bm  \psi(x)\|_2 \leq \sqrt{n}$ for all $x\in\Omega$. For this part we refer to \cite[Prop.\ 9]{KPUU24}, which yields this result (with $\varepsilon = 0$ in our context). 
    It remains to argue that $\|\bm{\psi}(x)\|_2 = \sqrt{n}$ on $\operatorname{supp} \lambda$. First note that \eqref{eq101} can be rewritten
    in the form
    \begin{equation*}
\bm{a}^{\ast} \Big( \sum\limits_{i=1}^{N+1} \lambda_i \bm{\psi}(x_i)\cdot \bm{\psi}(x_i)^{\ast} \Big) \bm{a} 
        = \|\bm{a}\|_2^2 \,.
    \end{equation*}
    As a consequence, the sum in the brackets is the identity matrix which has trace $n$. 
    We thus get
    \begin{equation*}
       n= \Tr \Big(\sum\limits_{i=1}^{N+1} \lambda_i \bm{\psi}(x_i)\cdot \bm{\psi}(x_i)^{\ast} \Big) = \sum\limits_{i=1}^{N+1} \lambda_i \|\bm{\psi}(x_i)\|_2^2 \,. 
    \end{equation*}
    Since $\sum_{i=1}^{N+1} \lambda_i = 1$, $\lambda_i>0$ for $x_i\in\operatorname{supp} \lambda$, and $\|{\bm \psi}(x)\|_2 \leq \sqrt{n}$ for all $x\in\Omega$, 
we obtain that $\|\bm{\psi}(x_i)\|_2 = \sqrt{n}$ for $x_i\in\operatorname{supp} \lambda$. 
\end{proof} 

\noindent\begin{algorithm}[t] \caption{\texttt{Construction of an equal norm tight frame}}\label{algo2}
    \begin{tabularx}{\textwidth}{lp{100pt}X}
        \textbf{Input:} & $\Omega$ & compact topological space \\
        & $\varphi_1, \dots, \varphi_n\colon \Omega\to\mathds C$ & linearly independent continuous functions on $\Omega$ \\[2pt]\hline\\[-8pt]
        \textbf{Output:} &   $(x_i)_i$ and $(\lambda_i)_i$    &  satisfying \eqref{eq101} and $\lambda_i>0$
        \\
        &  $\bm{\psi}(x)$ &  such that $\|\bm{\psi}(x_i)\|_2 = \sqrt{n}$ and $\|\bm{\psi}(x)\|_2 \leq \sqrt{n}$ for any $x\in \Omega$
        \\[2pt]\hline
    \end{tabularx}
    \begin{algorithmic}[1]
        \STATE{
$\bm{\varphi}(x) \coloneqq (\varphi_1(x), \dots, \varphi_n(x))^\top$ }
        \STATE{ put $\mathcal M \coloneqq \{\bm{\varphi}(x)\cdot \bm{\varphi}(x)^\ast  : x\in \Omega\}$ }
        \STATE{ put $N \coloneqq \dim(\Span\{\varphi_k\overline\varphi_l\}_{k,l=1}^{n})$ }
        \STATE{
        find $x_1, \dots, x_{N+1}\in \Omega$ and $\lambda_1, \dots, \lambda_{N+1}\ge 0$ with $\lambda_1+\dots+\lambda_{N+1}=1$ such that
        \begin{align*}
        \det\Big(\sum_{i=1}^{N+1}\lambda_i \bm{\varphi}(x_i)\cdot \bm{\varphi}(x_i)^\ast\Big)
        = \max_{\bm M\in\conv \mathcal M} \det(\bm M)
        \end{align*}
          }
        \STATE{ put $\bm{A}_{\lambda} = \sum_{i=1}^{N+1} \lambda_i \bm{\varphi}(x_i)\cdot \bm{\varphi}(x_i)^*$ }
        \STATE{ put
        $\bm{\psi}(x) := \mathbf{A}_{\lambda}^{-1/2}\cdot \bm{\varphi}(x)$ for $x\in \Omega$ }
        \STATE{ filter out the indices $i$ with $\lambda_i>0$; reindex accordingly }
        \RETURN{$\bm{\psi}(x)$; $(x_i)_i$ and $(\lambda_i)_i$ for indices $i$ with $\lambda_i>0$}
    \end{algorithmic}
\end{algorithm} 

\begin{remark} \Cref{frame_statement} can be adapted to the setting of \Cref{rem:finite}, i.e., linearly independent functions $\varphi_1,\ldots,\varphi_n:\Omega\to \mathds{C}$ on a
    finite set $\Omega$, in a straight-forward manner.
\end{remark} 

\begin{remark} The property $\|\bm{\psi}(x)\|_2 \le \sqrt{n}$ yields
\eqref{Sup_Estimate}. Indeed, for any $f\in V_n$ with $f(x)=\langle \bm{a}, \bm{\psi}(x)\rangle$ we have the estimate
\begin{equation*}
|f(x)|^2 = |\langle \bm{a}, \bm{\psi}(x)\rangle|^2  \leq \|\bm{a}\|^2_2 \|\bm{\psi}(x)\|^2_2 \le n \sum\limits_{i=1}^{N+1} \lambda_i|\langle \bm{a}, \bm{\psi}(x_i) \rangle|^2 =  n \sum\limits_{i=1}^{N+1} \lambda_i|f(x_i)|^2  \,.
\end{equation*}
\end{remark} 

\begin{remark}\label{rem:ENTF} The procedure described in \Cref{frame_statement} reminds on the way of making a given frame $\bm{\varphi}$ tight by multiplying with $\bm{A}^{-1/2}$, where $\bm{A} = \bm{T}_{\bm{\varphi}}^*\cdot \bm{T}_{\bm{\varphi}}$ denotes the frame matrix of $\bm{\varphi}$ and $\bm{T}_{\bm{\varphi}}$ its analysis operator. However, the transformed frame $\bm{\psi} = \bm{A}^{-1/2}\cdot \bm{\varphi}$ then usually does not consist of equal norm frame elements.
    A subsequent normalization and corresponding reweighting would yield an equal norm
    tight frame of type \eqref{eq101} or \eqref{tight_frame} with respect to a discrete probability measure.
    However, the separate normalization step would typically change the coefficient space such that $\Im(\bm{T}_{\bm{\psi}}) \neq \Im(\bm{T}_{\bm{\varphi}})$. In \Cref{frame_statement} the weights and points are constructed simultaneously and result in a discrete measure and a corresponding equal norm tight frame $\bm{\psi}$ where $\Im(\bm{T}_{\bm{\psi}}) = \Im(\bm{T}_{\bm{\varphi}})$.
    This construction thus maintains a stronger connection to the original frame. 
    Let us emphasize that there is significant interest in the frame community for the construction of unit/equal norm tight frames (ENTF,UNTF), see \cite{CFM12,CK03}.
\end{remark} %
 \section{Application: Exact Marcinkiewicz-Zygmund inequalities and quadrature}\label{sec:MZ_inequalities}

Frames have a related structure to $L_2$-MZ inequalities and tightness in frames means exactness in $L_2$-MZ inequalities. With that, we immediately obtain a method for constructing points and weights fulfilling the identity in the following theorem.

\begin{theorem}\label{exactmz} Let $\mu$ be a Borel probability measure on a compact topological space $\Omega$ and $V_n$ be an $n$-dimensional subspace 
    of $C(\Omega)$, the space of continuous complex-valued functions on $\Omega$. Then there exists an $N\le \dim(\Span\{f\cdot \overline{g} : f,g\in V_n\}) \le n^2$ and $N+1$ points $x_1,\dots,x_{N+1}\in \Omega$ together with non-negative weights satisfying $\sum_{i=1}^{N+1} \mu_i = 1$ such that 
    \begin{equation}\label{exactmz_formula}
        \int_\Omega |f(x)|^2 \;\mathrm d\mu(x) = \sum_{i=1}^{N+1} \mu_i |f(x_i)|^2
        \quad\text{for all}\quad
        f\in V_n \,.
    \end{equation}
\end{theorem} 

\begin{proof} Choose an orthonormal basis in $V_n$. By \Cref{discr_orth} this basis is also discretely orthonormal. The statement is obtained by Parseval's identity.  
\end{proof} 

The respective points and weights (the discrete measure) are determined by finding a maximizer through \eqref{Aalpha} taking Carathéodory subsampling into account. Exact $L_2$-MZ inequalities have a useful property. They also discretize (subsample) the corresponding inner product as the following corollary shows.

\begin{corollary}\label{product_of_functions} Let $\mu$ be a Borel probability measure on a compact topological space $\Omega$ and $V_n$ be an $n$-dimensional subspace of the complex-valued $C(\Omega)$. Then there exists an $N\leq n^2$ and $N+1$ points $x_1,\dots,x_{N+1}\in \Omega$ together with non-negative weights satisfying $\sum_{i=1}^{N+1} \mu_i = 1$ such that for all $f, g \in V_n$ it holds
    \begin{equation*}\label{statement}
        \int_\Omega f(x) \overline{g(x)} \;\mathrm d\mu(x)
        = \sum_{i=1}^{N+1} \mu_i f(x_i) \overline{g(x_i)} \,.
    \end{equation*}
\end{corollary} 

\begin{proof} Due to \Cref{exactmz}, for any $f, g \in V_n$ the relations \eqref{exactmz_formula} hold.
    Hence, they also hold for $f+g$, $f-g$, $f+i g$ and  $f-ig$ that belong to the subspace $V_n$.
    Using the polarization identity
    \begin{align*}
        \int_\Omega f(x)\overline{g(x)} \;\mathrm d\mu(x)
        &= \frac{1}{4} \left(  \int_\Omega |f(x)+g(x)|^2 \;\mathrm d\mu(x)  -   \int_\Omega |f(x)-g(x)|^2 \;\mathrm d\mu(x)\right)\\
        &\quad+ \frac{\mathrm i}{4} \left(  \int_\Omega |f(x)+\mathrm ig(x)|^2 \;\mathrm d\mu(x)  -   \int_\Omega |f(x)-\mathrm ig(x)|^2 \;\mathrm d\mu(x)\right) \,, 
    \end{align*}
    we then obtain
    \begin{align*}
        \int_\Omega f(x)\overline{g(x)} \;\mathrm d\mu(x)
        &=  \frac{1}{4} \left(  \sum_{i=1}^{N+1} \mu_i |f(x_i)+g(x_i)|^2  -   \sum_{i=1}^{N+1} \mu_i |f(x_i)-g(x_i)|^2 \right) \\
        &\quad+  \frac{\mathrm i}{4} \left( \sum_{i=1}^{N+1} \mu_i |f(x_i)+\mathrm ig(x_i)|^2 - \sum_{i=1}^{N+1} \mu_i |f(x_i)-\mathrm ig(x_i)|^2 \right) \\
        &= \sum_{i=1}^{N+1} \mu_i f(x_i)\overline{g(x_i)} \,.\qedhere
    \end{align*}
\end{proof} 

We obtain general quadrature formulas in the sense of Tchakaloff using our approach.
See Tchakaloff \cite{Tschak57} and also \cite{Pu97}, \cite[Thm.\ 4.1]{DaPrTeTi19}, which works for the case of real functions and has the drawback of not being constructive.

\begin{theorem}[Tschakaloff for complex-valued functions]\label{quadrature_formula}
    Let $\mu$ be a Borel probability measure on a compact topological space $\Omega$ and $V_n\subset C(\Omega)$ (complex-valued). Then there exists an $N\leq 2n$ and $x_1,\dots,x_{N+1}\in \Omega$ such that with non-negative weights satisfying $\mu_i+\dots+\mu_{N+1} = 1$ it holds
    \begin{equation*}
        \int_\Omega f(x) \;\mathrm d\mu(x) = \sum_{i=1}^{N+1} \mu_i f(x_i)
        \quad\text{for all}\quad
        f\in V_n \,.
    \end{equation*}
\end{theorem} 

\begin{proof} \emph{Step~1.}
    Applying \Cref{product_of_functions} to $V_n\cup\{\mathds 1_\Omega\}$ and setting $g=\mathds 1_\Omega$ gives the stated quadrature condition with $M \le (n+1)^2$ points $y_1, \dots, y_{M+1}$ and weights $\beta_1, \dots, \beta_{M+1}$.

    \emph{Step~2.}
    In order to obtain the stated number of points, we need to apply a second Carathéodory step.
    For $\varphi_1, \dots, \varphi_n$ a basis of $V_N$, the points and weights from Step~1 fulfill $\bm A(\beta_1 \, \cdots \, \beta_{M+1})^\top = \bm b$ with
    \begin{equation*}
        \bm A =
        \begin{pmatrix}
            \Re(\varphi_1(y_1)) & \dots & \Re(\varphi_1(y_{M+1})) \\
            \Im(\varphi_1(y_1)) & \dots & \Im(\varphi_1(y_{M+1})) \\
            \vdots & \ddots & \vdots \\
            \Re(\varphi_n(y_1)) & \dots & \Re(\varphi_n(y_{M+1})) \\
            \Im(\varphi_n(y_1)) & \dots & \Im(\varphi_n(y_{M+1})) \\
        \end{pmatrix}
        \in\mathds R^{2n\times M+1}
        \quad\text{and}\quad
        \bm b =
        \begin{pmatrix}
            \Re(\int_\Omega \varphi_1 \;\mathrm d\mu) \\
            \Im(\int_\Omega \varphi_1 \;\mathrm d\mu) \\
            \vdots \\
            \Re(\int_\Omega \varphi_n \;\mathrm d\mu) \\
            \Im(\int_\Omega \varphi_n \;\mathrm d\mu) \\
        \end{pmatrix}
        \in\mathds R^{2n} \,.
    \end{equation*}
    We interpret this as $\bm b$ being represented by a convex combination of the column vectors of $\bm A$.
    Applying Carathéodory here gives $N \le 2n$, non-negative real weights $\mu_1, \dots, \mu_{N+1}$, and a subset of points $\{x_1, \dots, x_N\}\subset\{y_1, \dots, y_M\}$ such that we have in $\mathds{C}^n$ 
    $$
        \sum\limits_{i=1}^{N+1} \mu_i (\varphi_1(x_1),...,\varphi_n(x_n))^\top = \Big(\int_\Omega \varphi_1 \;\mathrm d\mu,...,\int_\Omega \varphi_n \;\mathrm d\mu\Big)^\top\,.
    $$
    The statement in \Cref{quadrature_formula} follows by linearity since any $f$ can be represented as complex linear combination in the functions $\varphi_1,...,\varphi_n$.
\end{proof} 

Next, we show a method for constructing exact $L_p$-MZ inequalities with even $p$ by using similar techniques as in \cite[Theorem~3.1]{DaPrTeTi19} but utilizing the exact $L_2$ result in \Cref{exactmz} instead of Tchakaloff's Theorem on exact quadrature.

\begin{corollary}\label{even_MZ} Let $\mu$ be a Borel probability measure on a compact topological space $\Omega$ and $V_n\subset C(\Omega)$ (complex-valued) and $p$ even.
    Then there exists an $N\leq 2\binom{n+p-1}{p} \sim n^p$ and $x_1,\dots,x_{N+1}\in \Omega$ with non-negative weights satisfying $\mu_i+\dots+\mu_{N+1} = 1$ such that
    \begin{equation}
        \int_\Omega |f(x)|^p \;\mathrm d\mu(x) = \sum_{i=1}^{N+1} \mu_i |f(x_i)|^p
        \quad\text{for all}\quad
        f\in V_n \,.\label{eq:Lp_MZ_discretization}
    \end{equation}
\end{corollary} 

\begin{proof} Using a basis of $V_n = \Span\{\varphi_1, \dots, \varphi_n\}$ we define the space
    \begin{equation*}
        W_n
        = \Span\Big\{ \varphi_1^{k_1} \cdots \varphi_n^{k_n} : k_1, \dots, k_n\in \mathds N_0,\, k_1+\dots+k_n = \frac{p}{2} \Big\} \,.
    \end{equation*}
    We have
    \begin{align*}
        \dim\Big(\Span\{f\cdot g : f,g\in W_n\}\Big)
        &= \dim\Big(\Span\Big\{ \varphi_1^{k_1} \cdots \varphi_n^{k_n} : k_1,\dots,k_n\in \mathds N_0,\, k_1+\dots+k_n = p \Big\}\Big) \\
        &\le \Big| \Big\{ k_1,\dots,k_n\in \mathds N_0 : k_1+\dots+k_n = p \Big\} \Big| \\
        &= \binom{n+p-1}{p} \,,
    \end{align*}
    where the last equality is known as the number of weak compositions of the integer $p$ into $n$ parts.
    
    Applying \Cref{exactmz} to $W_n$ gives an $N \le \binom{n+p-1}{p}$ with points $x_1, \dots, x_{N+1}\in \Omega$ and non-negative weights satisfying $\mu_i+\dots+\mu_{N+1} = 1$ such that
    \begin{equation*}
        \int_\Omega |f(x)|^2 \;\mathrm d\mu(x) = \sum_{i=1}^{N+1} \mu_i |f(x_i)|^2
        \quad\text{for all}\quad
        f\in W_n \,.
    \end{equation*}
    In particular, it holds for all $f^{p/2}$ with $f\in V_n$, which is the assertion.
\end{proof} 

\Cref{even_MZ} contributes to \cite[Open~problem~10]{DaPrTeTi19} in the sense that it gives a constructive method to obtain an $L_p$-MZ inequality for even $p$.
Note, constructiveness here is meant in the sense of \cite[Section~2.5]{DaPrTeTi19} where a greedy algorithm is said to be constructive.
Since we have positive weights, it goes even further to make \cite[Corollary~4.2]{DaPrTeTi19} constructive.

Note, in general the number of points is optimal, cf.\ \cite[Theorem~3.2]{DaPrTeTi19}.
The condition of $p$ being an even number is necessary as it has been shown that otherwise such inequalities do not exist in general, cf.\ \cite[Prop.\ 3.3]{DaPrTeTi19}.

\section{Specific examples}\label{Sect:examples} 

\subsection{The $d$-torus} 

Let $\mathbb{T}=[0, 1)$ be a torus where the endpoints of the interval are identified. By 
$\mathbb{T}^d$ we denote a $d$-dimensional torus and equip it with the normalized Lebesgue measure 
${\rm d} \bm{x}$.
Let further $L_2:=L_2(\mathbb{T}^d, 
{\rm d} \bm{x})$ and 
$V
= V(I)
\coloneqq \operatorname{span}\{ \exp (
2\pi
\mathrm{i} \langle\bm k, \bm x\rangle) \colon 
\bm{k}\in I\subset \mathds{Z}^d, \ \bm{x}\in \mathbb{T}^d \}$. In what follows, we assume that $|I|$ is finite. Let us denote with $\mathcal{D}(I) := \{\bm{k}-\bm{\ell}~:~\bm{k},\bm{\ell}\in I\}$\,. 

\begin{corollary}\label{exactmz_torus} There exist points $\bm{x}^1,\dots, \bm{x}^{N} \in \mathbb{T}^d$ and non-negative weights $w_1, \dots, w_N$ satisfying $w_1+...+w_N = 1$ with $N = |\mathcal{D}(I)| \leq |I|^2-|I|+1$ such that for all $f\in V$ it holds
    \begin{equation}\label{torus_mz_inequality}
        \int_{\mathbb{T}^d} |f(\bm{x})|^2  \;{\mathrm d} \bm{x}
        = \sum_{i=1}^{N} w_i |f(\bm{x}^i)|^2 \,.
    \end{equation}
\end{corollary} 

\begin{proof} The statement with $N+1$ instead of $N$ follows immediately from \Cref{exactmz} by the fact, that $\dim(\Span\{f\cdot \overline{g} : f,g\in V\}) = |\mathcal{D}(I)| \leq |I|^2 -|I|+1$. In order to reduce further to $N = |\mathcal{D(I)}|$ points we return to the characterization in \Cref{discr_orth}. Recall, that for the reduction to $N+1$ we used Caratheodory's theorem in \Cref{CarSub}. This time we interprete the convex combination of the identity matrix $\bm{I}_n \in \conv \mathcal{M}$ as a conical combination (just non-negative coefficients which do not necessarily sum up to one). With \Cref{CarSub}, (i) we obtain \eqref{torus_mz_inequality} initially with $N$ non-negative weights $w_i$. However, since the space $V(I)$ contains complex monomials with modulus one everywhere we obtain $w_1+...+w_N = 1$ by testing \eqref{torus_mz_inequality} with such a monomial.
\end{proof} 

Next, we compare \Cref{exactmz_torus} to existing results in the torus setting. Let us emphasize that existing exact MZ-discretization results need additional restrictions on the index set. This is not the case in \Cref{exactmz_torus}. However, existing results use equal weights. This problem turns out to be much harder.   

\begin{itemize}
\item \emph{Equidistant points.}
    It is known that if $\sqrt[d]{n} \in \mathds{N}$ and $I\subset \{ -\sqrt[d]{n}/2 , \dots, \sqrt[d]{n}/2 -1\}^d$, then the grid  $\bm{G}$
    of equidistant points $\bm{X} := \{\bm{k}  / \sqrt[d]{n}\colon  \bm{k}\in \{ 1,\dots,  \sqrt[d]{n}\}^d\}$ with equal weights $1/n$ satisfies the exact Marcinkiewicz-Zygmund inequality, i.e.,
    for all $f\in V$
    it holds
    \begin{equation}\label{rank1}
        \int_{\mathbb{T}^d} |f(\bm{x})|^2  \;\mathrm d \bm{x} = 
        \frac{1}{n}\sum_{\bm{x}\in \bm{X}}  |f(\bm{x})|^2 \,.
    \end{equation}
    
    Equidistant points are an example of an exact $L_2$-MZ inequality with equal weights. If $I$ is small compared to the frequency cube $\{ -\sqrt[d]{n}/2 , \dots, \sqrt[d]{n}/2 -1\}^d$ this setting is not very useful in high dimensions $d$. However, if $I$ denotes the entire frequency cube, no oversampling is needed, i.e., $|I| = |\bm{G}| = n$.
    
\item \emph{Rank-1 lattices \cite{Kammerer_Potts_Volkmer_2015} and quadratic oversampling.}
    For a set $I\subset \mathds{Z}^d$, 
    \begin{equation}\max_{j=1,\dots,d}\max_{\bm{h},\bm{k}\in I}|h_j-k_j|<|I|,\label{eq:R1L_add_restr}
    \end{equation}
    there is some $M\in \mathds{N}$ satisfying $|I|\le M \le |\mathcal{D}(I)| \le |I|^2-|I|+1$ such that one can construct a rank-1 lattice
    $\bm{X} := \{\frac{i}{M}\bm{z} \operatorname{mod} \mathds 1\colon  i=0, \dots, M-1\}$ such that for all $f\in V(I)$ \eqref{rank1} holds. Moreover, such a rank-1 lattice can be efficiently determined using a component-by-component approach, cf.~\cite{KMNN21}.
    
    It is worth mentioning that rank-1 lattices  have a similar quadratic oversampling as in \Cref{exactmz_torus}. In \cite{Te17} it is claimed that the additional restriction \eqref{eq:R1L_add_restr} on the index set can be dropped. In fact, it is claimed that a Korobov lattice of size $M \leq 2d|\mathcal{D}(I)|$ always exists, such that \eqref{rank1} is true for all $f \in V(I)$. A closer inspection of the proof in \cite[Sect.\ 4]{Te17}, however, shows that an additional assumption similar to \eqref{eq:R1L_add_restr} is used in the proof.  
    \begin{theorem}\label{thm:rationals_aliasing}
         Let $M \in \mathds{N}$, $\bm{R}_K = \{0,1/K,...,(K-1)/K\}$ the full one-dimensional grid with $K$ equidistant points and $\bm{R}^{M,d}=\left(\bigcup_{K=1}^M \bm{R}_K\right)^d$ a $d$-dimensional rectilinear grid. Then there always exist index sets $I \subset \mathds{Z}^d$ with two elements such that for all $\bm{X} \subset \bm{R}^{M,d}$ the exact MZ inequality \eqref{torus_mz_inequality} is violated.
    \end{theorem}

    \begin{proof}
        We simply specify a range of sets that fulfills the assertion. Let
       \begin{equation*}
           I_{\bm a,\bm b}:=\left\{\bm a, \bm a+\left(\prod_{k=1}^{M}k\right)\bm{b}\right\}, \qquad \bm a\in\mathds{Z}^d, \bm b\in\mathds{Z}^d\setminus\{\bm 0\},
       \end{equation*}
       Clearly, for any $\bm x\in \bm{R}^{M,d}$ and $j=1,\dots,d$, there exist $K_j\in\mathds{N}\cap[0,M]$ and $l_j\in\mathds{N}\cap[0,K_j-1]$ such that $x_j=l_j/K_j$, which yields
       \begin{align*}
            \exp \left(2\pi \mathrm{i} \langle\bm a, \bm x \rangle\right)
            &=\prod_{j=1}^d\exp \left(2\pi \mathrm{i}\, a_j x_j\right)
            =\prod_{j=1}^d\exp \left(2\pi \mathrm{i}\, a_j \frac{l_j}{K_j}\right).
        \intertext{Due to the fact that $\frac{l_j}{K_j}\left(\prod_{k=1}^{M}k\right)b_j\in\mathds{N}_0$ for each $j$, we observe}
        \exp \left(2\pi \mathrm{i} \langle\bm a, \bm x \rangle\right)&=
        \prod_{j=1}^d\left(\exp \left(2\pi \mathrm{i}\, a_j \frac{l_j}{K_j}\right)\exp \left(2\pi \mathrm{i}\, \frac{l_j}{K_j}\left(\prod_{k=1}^{M}k\right)b_j \right)\right)
        \\&=
        \exp \left(2\pi \mathrm{i} \left\langle\bm a+\left(\prod_{k=1}^{M}k\right)\bm b, \bm x \right\rangle\right).
       \end{align*}
    
       As a consequence, for the function $f\colon\mathbb{T}^d\to \mathds{C}$, $\bm x\mapsto \exp \left(2\pi \mathrm{i} \langle\bm a, \bm x\rangle\right)-\exp \left(2\pi \mathrm{i} \langle\bm a', \bm x\rangle\right)$, $\bm a'=\bm a+\left(\prod_{k=1}^{M}k\right)\bm{b}$ ,we observe $f\in V(I_{\bm a,\bm b})$ and that the right hand side of \eqref{torus_mz_inequality} is zero independent of $\bm{X}\subset \bm{R}^{M,d}$, while the left hand side is not.
    \end{proof}
    
    \begin{corollary}\label{cor:lattice_aliasing}
        For any $n \in \mathds{N}$, $n\ge 2$, and $M \in \mathds{N}$ there are index sets $I\subset\mathds{Z}^d$, $|I| = n$, such that \eqref{torus_mz_inequality} is violated for any lattice rule with not more than M elements.   
    \end{corollary}
    \begin{proof}
        First we assume the lattice rule being a rank-$m$ lattice rule with $M'\le M$ sampling nodes $\bm{X}$. Moreover, we assume that an expression of the lattice rule is given in canonical form, cf.\ \cite[Theorem 4.5]{SlLy89}, which implies that each used sampling node $\bm x\in \bm{X}$ can be written as
        $$
        \bm x=\frac{j_1\bm z_1}{M_1'}+\cdots+\frac{j_m\bm z_m}{M_m'}\operatorname{mod} \mathds 1\,,
        $$
        where $j_1,\dots,j_m\in\mathds{N}_0$, $\bm z_1,\dots,\bm z_m\in\mathds{Z}^d$, and $M'=\prod_{j=1}^mM_j'$. Obviously, each component of each sampling node $\bm x\in \bm{X}\subset [0,1)^d$ can be written as a rational with denominator at most $M'$, which means that $\bm{X}\subset \bm{R}^{M',d}\subset \bm{R}^{M,d}$ and we can apply \Cref{thm:rationals_aliasing} and subsequently add arbitrary indices from $\mathds{Z}^d$ to $I$ until $|I| = n$ is reached.
    \end{proof}
    
    \begin{remark} The index sets and fooling functions constructed in the proofs of \Cref{thm:rationals_aliasing} and \Cref{cor:lattice_aliasing} can also be applied to \eqref{eq:Lp_MZ_discretization} with $\Omega=\mathbb{T}^d$, $V_n=V(I)$, and $d\mu(x)$ the normalized Lebesgue measure.
        Consequently, sampling sets consisting of rational sampling nodes with bounded denominators do not provide exact $L_p$-MZ discretizations for arbitrary spans $V(I)$, $|I|<n$, of trigonometric monomials. In particular, lattice rules of bounded cardinality cannot fulfill this characteristic either.
    \end{remark} 

\item \emph{Random points \cite{Tropp_2011} and logarithmic oversampling.}
    For any $\varepsilon >0$ there exist $M \leq \frac{6}{\varepsilon^2} |I| \log |I|$ uniformly drawn points
    $\bm{x}^1, \dots, \bm{x}^M$
    such that with high probability it holds for all $f\in V$ the tight bound
    \begin{equation}\label{eq102}
        (1-\varepsilon) \int_{\mathbb{T}^d} |f(\bm{x})|^2  \;\mathrm d \bm{x}  \leq 
        \frac{1}{M} \sum_{i=1}^{M}  |f(\bm{x}^i)|^2 \leq 
        (1+\varepsilon) \int_{\mathbb{T}^d} |f(\bm{x})|^2  \;\mathrm d \bm{x}\, .
    \end{equation}

    Random points show the existence of $L_2$-MZ inequalities with logarithmic instead of quadratic oversampling but loose the exactness condition.

\item \emph{Constant oversampling \cite{BaSpSr12},\cite{Te17},\cite{BSU23}.}
    Let $b>1$ and $M= \lceil b |I| \rceil$. Then there exist subsampled random points
    $\bm{x}^1,\dots, \bm{x}^{M} \in \mathbb{T}^d$ and positive weights $w_1, \dots, w_M$
    such that
it holds for all $f\in V$ that
    \begin{equation*}
        \int_{\mathbb{T}^d} |f(\bm{x})|^2  \;\mathrm d \bm{x}  \leq 
        \sum_{i=1}^{M} w_i  |f(\bm{x}^i)|^2 \leq 
        \left(\frac{\sqrt{b}+1}{\sqrt{b}-1}\right)^2 \int_{\mathbb{T}^d} |f(\bm{x})|^2 \;\mathrm d \bm{x}\, .
    \end{equation*}
    Here the question arises if constant weights can be used here. Based on the groundbreaking works by Marcus, Spielman, Srivastava \cite{MaSpSr15} and Nitzan, Olevskii, Ulanovskii \cite{NiOlUl16} one should mention the following recent result by Kosov \cite[Cor.\ 1.3]{Ko22}. If $\varepsilon>0$ then there is a set of $M$ points $\bm{x}^1,...,\bm{x}^M \in \mathbb{T}^d$ with $M\leq 10^5\varepsilon^{-2}|I|$ such that \eqref{eq102} holds.  All these  results show discretization inequalities for the square norm with merely linear oversampling. The last one even gives a tight bound with $\varepsilon$-distortion. 
\end{itemize}

\subsection{The $d$-sphere} Let $\mathds S^d = \{\bm x\in\mathds R^{d+1} : \|\bm x\| = 1\}$ be the unit sphere with dimension $d$ and let $\mu(\bm x)$ be the normalized surface measure.
We have a look at the space of polynomials of degree at most $m$ restricted to the sphere
\begin{equation*}
    \Pi_m
    \coloneqq \Span\Big\{ f:\mathds S^d \to \mathds R ~:~ f(\bm x) = \bm{x}^{\bm{k}} = x_1^{k_1}\cdots x_d^{k_d}, \bm k\in \mathds N_0^d,\, \|\bm{k} \|_1 \le m \Big\} \,.
\end{equation*}
These are commonly used in approximation on the sphere as they can be represented by spherical harmonics allowing for applying analytic tools and fast algorithms, cf.\ \cite{DX13}.
For the dimension of this space, we have by \cite[Corollary~1.1.5]{DX13}
\begin{equation*}
    \dim \Pi_m
    = \binom{m+d}{d} + \binom{m+d-1}{d}
    \le \Big(\frac{9m}{d}\Big)^d \,,
\end{equation*}
where $d\le m$ and $\binom{n}{k} \le (n\cdot e/k)^k$ was used in the last inequality.

\begin{theorem}\label{exactmz_sphere} Let $m\ge d$, $p\in\mathds N$ even.
    There exist points $\bm x^1,\dots,\bm x^N$ in $\mathds{S}^d$ and positive weights $w_1+\dots+w_N=1$ with $N\le \dim(\Pi_{pm}) \le (9pm/d)^d$ such that
    \begin{equation*}
        \int_{\mathds{S}^d} |f(\bm x)|^p \;\mathrm d\mu(\bm x) = \sum_{i=1}^N w_i |f(\bm x^i)|^p
        \quad\text{for all}\quad
        f\in \Pi_m \,.
    \end{equation*}
\end{theorem} 

\begin{proof} The proof works analogous to \Cref{even_MZ}:
    Using a basis of $\Pi_m = \Span\{\varphi_1, \dots, \varphi_{\dim(\Pi_m)}\}$ we define the space
    \begin{equation*}
        W_m
        = \Span\Big\{ \varphi_1^{k_1} \cdots \varphi_{\dim(\Pi_m)}^{k_n} : k_1, \dots, k_n\in \mathds N_0,\, k_1+\dots+k_n = \frac{p}{2} \Big\} \,.
    \end{equation*}
    Because of the polynomial structure, we have $W_m = \Pi_{pm/2}$.
    It is left to apply \Cref{exactmz} to $\Pi_{pm/2}$ and use $\dim(\{f\cdot\overline g : f, g\in \Pi_{pm/2}\}) = \dim(\Pi_{pm}) \le (9pm/d)^d$.

    Because $\mathds 1_{\mathds S^d}\in\Pi_m$ for all $m$, the number of points can be reduced to $\dim(\Pi_{pm})$ instead of $\dim(\Pi_{pm})+1$ by the same reasoning as in \Cref{exactmz_torus} .
\end{proof} 

Next, we compare \Cref{exactmz_sphere} to existing results.

\begin{itemize}
\item
    For $\mathds S^2$ the assertion above gives the existence of an exact $L_2$-MZ inequality for $N\le (2m+1)^2$ points, which matches the lower bound of $N\ge m(m+1)/2$ in terms of the rate from \cite{DaPrTeTi19}.
\item
    Similar to the torus, one can draw uniform points $\{\bm x^1, \dots, \bm x^M\}\subset\mathds S^d$ with $M \ge C\dim(\Pi_m)\log(\dim\Pi_m) \varepsilon^{-2}$ at random to obtain an equal-weight, non-exact $L_2$-MZ inequality, i.e.,
    \begin{equation*}
        (1-\varepsilon)\|f\|_{L_2(\mu)}^{2}
        \le \frac{1}{M}\sum_{i=1}^{M}|f(\bm x^i)|^2
        \le (1+\varepsilon)\|f\|_{L_2(\mu)}^{2}
        \quad\text{for all}\quad
        f\in\Pi_m \,.
    \end{equation*}
    This result is stated in \cite[Theorem~3.2]{FiHiJaUl24}.
\item
    Furthermore, in \cite[Theorem~4.1]{FiHiJaUl24} the existence of $M$ points with $M \le C\dim(\Pi_m) \varepsilon^{-d}$ was shown which simultaneously fulfill an equal-weight, non-exact $L_p$-MZ inequality for $p\in [1,\infty]$.
\end{itemize}

Moving on to exact quadrature, we immediately obtain the following result.

\begin{theorem}\label{quadrature_sphere} Let $m\ge d$.
    There exist points $\bm x^1,\dots,\bm x^{N}$ in $\mathds{S}^d$ and positive weights $w_1+\dots+w_{N}=1$ with $N\le \dim(\Pi_{m}) \le (9m/d)^d$ such that
    \begin{equation*}
        \int_{\mathds{S}^d} f(\bm x) \;\mathrm d\mu(\bm x) = \sum_{i=1}^{N} w_i f(\bm x^i) 
        \quad\text{for all}\quad
        f\in \Pi_{m} \,.
    \end{equation*}
\end{theorem} 

\begin{proof} The proof works analogous to \Cref{quadrature_formula}.
    But since we have real-valued basis functions and $\mathds 1_{\mathds S^d}\in\Pi_m$ for all $m$, we use the conical variant of Carathéodory's Theorem
    (as in \Cref{exactmz_torus}) in order to obtain $\dim(\Pi_{m})$ points instead of $2\dim(\Pi_{m})+1$.
\end{proof} 

Quadrature formulas of the above type with equal weights are called $t$-designs introduced in \cite{Delsarte_Goethals_Seidel1977} (or $m$-design according to our notation, where $m$ is the degree of the polynomial). The equal weight condition makes this problem much harder and only a limited number of constructions of spherical designs are known.
However, there are also approaches to obtain spherical designs being exact up to machine precision numerically, cf.\ \cite{Womersley18, GP11}. Let us also mention \cite{ACSW10}, where the authors use a similar technique of maximizing a certain Gram determinant (for the case $d=2$) in order to computationally construct well-conditioned spherical design with a number of points $M\geq (m+1)^2$ points. In general the existence of spherical designs is known with the optimal asymptotic rate $M\ge C_d m^d$, 
cf.\ \cite[Thm.\ 1]{BRV13}.
This matches the number of points of our weighted result in \Cref{quadrature_sphere}. \section{Numerical experiments}\label{sec:numerics} 

In this section we test \Cref{algo1} numerically in order to obtain points and weights forming an exact $L_2$-MZ inequality.
For that let us assume a product type basis $\varphi_{\bm k}(\bm x) = \varphi_{k_1}(x_1) \cdots \varphi_{k_d}(x_d)$ and, as before, $\bm \varphi(x) = (\varphi_{\bm k}(\bm x))_{\bm k\in I}^{\top}$ for some multi-index set $I$.
For the minimization process we use the Broyden–Fletcher–Goldfarb–Shanno (BFGS) algorithm, which is an iterative quasi-Newton optimizer for unconstrained non-linear optimization problems, in order to update the points and weights in an alternating fashion.
It uses the objective function
\begin{align*}
    f(\bm x^1, \dots, \bm x^{N+1}, \alpha_1, \dots, \alpha_{N+1})
    =
    \Big\| \sum_{i=1}^{N+1} \alpha_i \bm \varphi(\bm x^i)\cdot \bm \varphi(\bm x^i)^\ast
    -\bm I_m \Big\|_F^2
    = \sum_{\bm k,\bm l\in I} \Big| \sum_{i=1}^{N+1} \alpha_i \varphi_{\bm k}(\bm x^i) \overline{\varphi_{\bm l}(\bm x^i)} - \delta_{\bm k\bm l} \Big|^2 \,.
\end{align*}
and the partial derivatives:
noting that for any complex numbers $z_1$ and $z_2$ it holds $\Re (z_1)\Re (z_2) + \Im (z_1)\Im (z_2) = \Re (z_1 \overline{z_2})$, they evaluate to
\begin{align*}
    \frac{\partial f}{\partial x_j^{i^\prime}}
    &= \sum_{\bm k,\bm l\in I}
    2\Re\Big( \sum_{i=1}^{N+1} \alpha_{i} \varphi_{\bm k}(\bm x^{i})\overline{\varphi_{\bm l}(\bm x^{i})} - \delta_{\bm k, \bm l} \Big)
    \Re\Big(\alpha_{i^\prime}  \varphi_{\bm k}(\bm x^{i^\prime}) \Big(\frac{\partial}{\partial x_j^{i^\prime}}\overline{\varphi_{ \bm l}(\bm x^{i^\prime})}\Big) + \alpha_{i^\prime} \Big(\frac{\partial}{\partial x_j^{i^\prime}}\varphi_{\bm k}(\bm x^{i^\prime})\Big) \overline{\varphi_{ \bm l}(\bm x^{i^\prime})} \Big) \\
    &\quad +2\Im\Big( \sum_{i=1}^{N+1} \alpha_{i}\varphi_{\bm k}(\bm x^{i})\overline{\varphi_{\bm l}(\bm x^{i})} - \delta_{\bm k, \bm l} \Big)
    \Im\Big( \alpha_{i^\prime} \varphi_{\bm k}(\bm x^{i^\prime}) \Big(\frac{\partial}{\partial x_j^{i^\prime}}\overline{\varphi_{ \bm l}(\bm x^{i^\prime})}\Big) + \alpha_{i^\prime} \Big(\frac{\partial}{\partial x_j^{i^\prime}}\varphi_{\bm k}(\bm x^{i^\prime})\Big) \overline{\varphi_{ \bm l}(\bm x^{i^\prime})} \Big) \\
    &= 4\alpha_{i^\prime} \Re\Big(
    \sum_{\bm k,\bm l\in I} \overline{\varphi_{\bm k}(\bm x^{i^\prime})} \Big( \sum_{i=1}^{N+1} \alpha_{i} \varphi_{\bm k}(\bm x^{i})\overline{\varphi_{\bm l}(\bm x^{i})} - \delta_{\bm k \bm l}\Big) \Big( \frac{\partial}{\partial x_j^{i^\prime}} \varphi_{\bm l}(\bm x^{i^\prime}) \Big)
    \Big) \\
    &= 4\alpha_{i^\prime} \Re \Big( \bm \varphi(\bm x^{i^\prime})^\ast \Big( \sum_{i=1}^{N+1} \alpha_{i} \bm \varphi(\bm x^{i})\cdot \bm \varphi(\bm x^{i})^\ast - \bm I_m \Big) \frac{\partial}{\partial x_j^{i^\prime}} \bm \varphi(\bm x^{i^\prime}) \Big)
\end{align*}
and
\begin{equation*}
    \frac{\partial f}{\partial \alpha_{i^\prime}}
    = 2 \Re \Big( \bm \varphi(\bm x^{i^\prime})^\ast \Big( \sum_{i=1}^{N+1} \alpha_{i} \bm \varphi(\bm x^{i})\cdot \bm \varphi(\bm x^{i})^\ast - \bm I_m \Big) \bm \varphi(\bm x^{i^\prime}) \Big) \,.
\end{equation*}

In our experiments, we use trigonometric polynomials $\varphi_{\bm k}(\bm x) = \exp(2\pi\mathrm i \langle\bm k, \bm x\rangle)$ for different frequency index sets $I\in\mathds Z^d$ and dimensions $d$.
The initial draw of points is made randomly and the weights are set to be equal.
The optimization procedure is done several times for different initial draws to counteract a bad draw of points.

\paragraph{Experiment~1} 

For the first experiments we use dimension $d=2$ and three different frequency index sets:
An $\ell_1$-ball with $|I_1|=41$ frequencies, a hyperbolic cross with $|I_2|=33$ frequencies, i.e.,
\begin{equation*}
    I_1
    = \Big\{\bm k\in\mathds Z^2 : \|\bm k\|_1 \le 4 \Big\}
    \quad\text{and}\quad
    I_2
    = \Big\{\bm k\in\mathds Z^2 : \prod_{j=1}^{2} (|k_j|+1) \le 6 \Big\}
\end{equation*}
and the third choice is based on the argumentations leading to \Cref{cor:lattice_aliasing} with $|I_3| = 10$, where we know that the minimal lattice size is $113$ for exact reconstruction and our theoretical bound is $\dim \{V\cdot\overline V\} = 91$:
\begin{align*}
    I_3 = \Big\{
    & \begin{pmatrix}         0 \\         0 \end{pmatrix},
    \begin{pmatrix}   2\,671\,704 \\   2\,671\,704 \end{pmatrix},
    \begin{pmatrix}  -3\,111\,990 \\   3\,111\,990 \end{pmatrix},
    \begin{pmatrix}  -4\,145\,974 \\  -4\,145\,974 \end{pmatrix},
    \begin{pmatrix}   4\,520\,742 \\  -4\,520\,742 \end{pmatrix}, \\
    & \begin{pmatrix}  -5\,553\,600 \\  -5\,553\,600 \end{pmatrix},
    \begin{pmatrix}  -6\,867\,835 \\   6\,867\,835 \end{pmatrix},
    \begin{pmatrix}  18\,119\,640 \\  18\,119\,640 \end{pmatrix}
    \begin{pmatrix}  39\,011\,940 \\ -39\,011\,940 \end{pmatrix},
    \begin{pmatrix} -39\,021\,892 \\  39\,021\,892 \end{pmatrix},
    \Big\} \,.
\end{align*}

As for the number of points, we seeked the smallest number such that we still have exact reconstruction up to a certain threshold, i.e., an $L_2$-MZ constant of $\|\sum_{i=1}^{N+1}\alpha_i \bm \varphi(\bm x^i)\cdot\bm \varphi(\bm x^i)^\ast - \bm I_m\|_{2\to 2} = \varepsilon < 10^{-13}$.

\begin{figure} \centering
    \begingroup
  \makeatletter
  \providecommand\color[2][]{\GenericError{(gnuplot) \space\space\space\@spaces}{Package color not loaded in conjunction with
      terminal option `colourtext'}{See the gnuplot documentation for explanation.}{Either use 'blacktext' in gnuplot or load the package
      color.sty in LaTeX.}\renewcommand\color[2][]{}}\providecommand\includegraphics[2][]{\GenericError{(gnuplot) \space\space\space\@spaces}{Package graphicx or graphics not loaded}{See the gnuplot documentation for explanation.}{The gnuplot epslatex terminal needs graphicx.sty or graphics.sty.}\renewcommand\includegraphics[2][]{}}\providecommand\rotatebox[2]{#2}\@ifundefined{ifGPcolor}{\newif\ifGPcolor
    \GPcolortrue
  }{}\@ifundefined{ifGPblacktext}{\newif\ifGPblacktext
    \GPblacktexttrue
  }{}\let\gplgaddtomacro\g@addto@macro
\gdef\gplbacktext{}\gdef\gplfronttext{}\makeatother
  \ifGPblacktext
\def\colorrgb#1{}\def\colorgray#1{}\else
\ifGPcolor
      \def\colorrgb#1{\color[rgb]{#1}}\def\colorgray#1{\color[gray]{#1}}\expandafter\def\csname LTw\endcsname{\color{white}}\expandafter\def\csname LTb\endcsname{\color{black}}\expandafter\def\csname LTa\endcsname{\color{black}}\expandafter\def\csname LT0\endcsname{\color[rgb]{1,0,0}}\expandafter\def\csname LT1\endcsname{\color[rgb]{0,1,0}}\expandafter\def\csname LT2\endcsname{\color[rgb]{0,0,1}}\expandafter\def\csname LT3\endcsname{\color[rgb]{1,0,1}}\expandafter\def\csname LT4\endcsname{\color[rgb]{0,1,1}}\expandafter\def\csname LT5\endcsname{\color[rgb]{1,1,0}}\expandafter\def\csname LT6\endcsname{\color[rgb]{0,0,0}}\expandafter\def\csname LT7\endcsname{\color[rgb]{1,0.3,0}}\expandafter\def\csname LT8\endcsname{\color[rgb]{0.5,0.5,0.5}}\else
\def\colorrgb#1{\color{black}}\def\colorgray#1{\color[gray]{#1}}\expandafter\def\csname LTw\endcsname{\color{white}}\expandafter\def\csname LTb\endcsname{\color{black}}\expandafter\def\csname LTa\endcsname{\color{black}}\expandafter\def\csname LT0\endcsname{\color{black}}\expandafter\def\csname LT1\endcsname{\color{black}}\expandafter\def\csname LT2\endcsname{\color{black}}\expandafter\def\csname LT3\endcsname{\color{black}}\expandafter\def\csname LT4\endcsname{\color{black}}\expandafter\def\csname LT5\endcsname{\color{black}}\expandafter\def\csname LT6\endcsname{\color{black}}\expandafter\def\csname LT7\endcsname{\color{black}}\expandafter\def\csname LT8\endcsname{\color{black}}\fi
  \fi
    \setlength{\unitlength}{0.0500bp}\ifx\gptboxheight\undefined \newlength{\gptboxheight}\newlength{\gptboxwidth}\newsavebox{\gptboxtext}\fi \setlength{\fboxrule}{0.5pt}\setlength{\fboxsep}{1pt}\definecolor{tbcol}{rgb}{1,1,1}\begin{picture}(7920.00,4520.00)\gplgaddtomacro\gplbacktext{\csname LTb\endcsname \put(489,2802){\makebox(0,0)[r]{\strut{}\scriptsize -4}}\csname LTb\endcsname \put(489,3375){\makebox(0,0)[r]{\strut{}\scriptsize 0}}\csname LTb\endcsname \put(489,3947){\makebox(0,0)[r]{\strut{}\scriptsize 4}}\csname LTb\endcsname \put(744,2454){\makebox(0,0){\strut{}\scriptsize -4}}\csname LTb\endcsname \put(1316,2454){\makebox(0,0){\strut{}\scriptsize 0}}\csname LTb\endcsname \put(1888,2454){\makebox(0,0){\strut{}\scriptsize 4}}}\gplgaddtomacro\gplfronttext{\csname LTb\endcsname \put(1316,4336){\makebox(0,0){\strut{}\scriptsize $\ell_1$-ball ($m = 41$)}}}\gplgaddtomacro\gplbacktext{\csname LTb\endcsname \put(3122,2778){\makebox(0,0)[r]{\strut{}\scriptsize -5}}\csname LTb\endcsname \put(3122,3375){\makebox(0,0)[r]{\strut{}\scriptsize 0}}\csname LTb\endcsname \put(3122,3971){\makebox(0,0)[r]{\strut{}\scriptsize 5}}\csname LTb\endcsname \put(3353,2454){\makebox(0,0){\strut{}\scriptsize -5}}\csname LTb\endcsname \put(3950,2454){\makebox(0,0){\strut{}\scriptsize 0}}\csname LTb\endcsname \put(4546,2454){\makebox(0,0){\strut{}\scriptsize 5}}}\gplgaddtomacro\gplfronttext{\csname LTb\endcsname \put(3949,4336){\makebox(0,0){\strut{}\scriptsize hyperbolic cross ($m = 33$)}}}\gplgaddtomacro\gplbacktext{\csname LTb\endcsname \put(5755,2802){\makebox(0,0)[r]{\strut{}\scriptsize $-4\cdot 10^7$}}\csname LTb\endcsname \put(5755,3375){\makebox(0,0)[r]{\strut{}\scriptsize $0$}}\csname LTb\endcsname \put(5755,3947){\makebox(0,0)[r]{\strut{}\scriptsize $4\cdot 10^7$}}\csname LTb\endcsname \put(6011,2454){\makebox(0,0){\strut{}\scriptsize $-4\cdot 10^7$}}\csname LTb\endcsname \put(6583,2454){\makebox(0,0){\strut{}\scriptsize $0$}}\csname LTb\endcsname \put(7155,2454){\makebox(0,0){\strut{}\scriptsize $4\cdot 10^7$}}}\gplgaddtomacro\gplfronttext{\csname LTb\endcsname \put(6583,4336){\makebox(0,0){\strut{}\scriptsize bad frequencies ($m = 10$)}}}\gplgaddtomacro\gplbacktext{\csname LTb\endcsname \put(448,410){\makebox(0,0)[r]{\strut{}\scriptsize 0}}\csname LTb\endcsname \put(448,1839){\makebox(0,0)[r]{\strut{}\scriptsize 1}}\csname LTb\endcsname \put(560,205){\makebox(0,0){\strut{}\scriptsize 0}}\csname LTb\endcsname \put(1989,205){\makebox(0,0){\strut{}\scriptsize 1}}}\gplgaddtomacro\gplfronttext{\csname LTb\endcsname \put(2208,410){\makebox(0,0)[l]{\strut{}\scriptsize 0.024}}\csname LTb\endcsname \put(2208,1839){\makebox(0,0)[l]{\strut{}\scriptsize 0.025}}\csname LTb\endcsname \put(1274,2085){\makebox(0,0){\strut{}\scriptsize $n=41, \varepsilon=1.60834\cdot 10^{-14}$}}}\gplgaddtomacro\gplbacktext{\csname LTb\endcsname \put(3081,410){\makebox(0,0)[r]{\strut{}\scriptsize 0}}\csname LTb\endcsname \put(3081,1839){\makebox(0,0)[r]{\strut{}\scriptsize 1}}\csname LTb\endcsname \put(3193,205){\makebox(0,0){\strut{}\scriptsize 0}}\csname LTb\endcsname \put(4622,205){\makebox(0,0){\strut{}\scriptsize 1}}}\gplgaddtomacro\gplfronttext{\csname LTb\endcsname \put(4841,410){\makebox(0,0)[l]{\strut{}\scriptsize 0.022}}\csname LTb\endcsname \put(4841,1839){\makebox(0,0)[l]{\strut{}\scriptsize 0.023}}\csname LTb\endcsname \put(3907,2085){\makebox(0,0){\strut{}\scriptsize $n=45, \varepsilon=1.08850\cdot 10^{-14}$}}}\gplgaddtomacro\gplbacktext{\csname LTb\endcsname \put(5714,410){\makebox(0,0)[r]{\strut{}\scriptsize 0}}\csname LTb\endcsname \put(5714,1839){\makebox(0,0)[r]{\strut{}\scriptsize 1}}\csname LTb\endcsname \put(5826,205){\makebox(0,0){\strut{}\scriptsize 0}}\csname LTb\endcsname \put(7255,205){\makebox(0,0){\strut{}\scriptsize 1}}}\gplgaddtomacro\gplfronttext{\csname LTb\endcsname \put(7474,410){\makebox(0,0)[l]{\strut{}\scriptsize 0.000}}\csname LTb\endcsname \put(7474,982){\makebox(0,0)[l]{\strut{}\scriptsize 0.010}}\csname LTb\endcsname \put(7474,1553){\makebox(0,0)[l]{\strut{}\scriptsize 0.020}}\csname LTb\endcsname \put(6541,2085){\makebox(0,0){\strut{}\scriptsize $n=91, \varepsilon=3.62941\cdot 10^{-16}$}}}\gplbacktext
    \put(0,0){\includegraphics[width={396.00bp},height={226.00bp}]{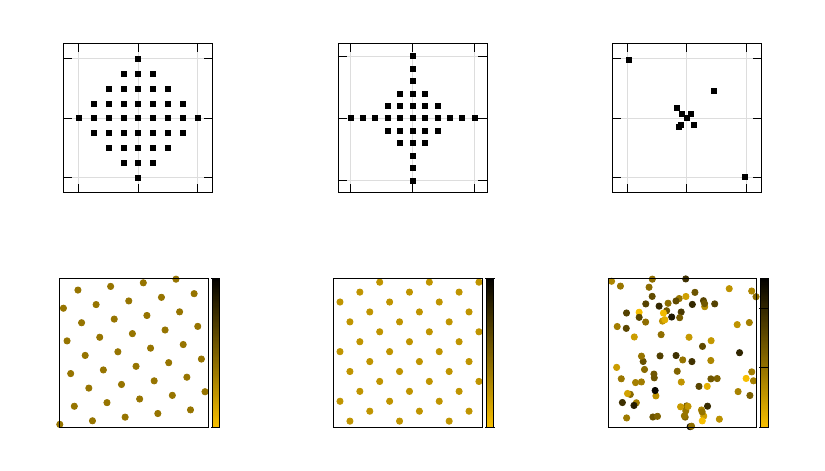}}\gplfronttext
  \end{picture}\endgroup
     \caption{Experiment~1: points and weights forming an exact $L_2$-MZ inequality for frequencies in an $\ell_1$-ball (left), a hyperbolic cross (middle), and $2$-dimensional bad frequencies (right).}\label{fig:points}
\end{figure} 

The outcome is depicted in \Cref{fig:points}.
For the $\ell_1$-ball and the hyperbolic cross we obtain very regular sets of sampling nodes with almost equal weights, reminiscent of shifted rank-1 lattices.
In particular, we achieve an oversampling factor of one for the $\ell_1$-ball, i.e., $n = |I_1| = 41$, where also a reconstructing lattice is known.
The third frequency set $I_3$ appears to be more challenging as we were not able to find points and weights for $n<91$, which suggests the sharpness of our results.
It is also an example of non-equal weights.
We further tried several strategies to obtain an equal-weight exact $L_2$-MZ inequality by only optimizing the points but did not succeed with our algorithm.

\paragraph{Experiment~2} 

Now we want to explore a possible dimension-dependence of the needed number of points for an exact $L_2$-MZ inequality.
We chose $20$ random frequencies in $\{-100, \dots, 100\}^d$ for $d=1,\dots, 7$.
We applied \Cref{algo1} varying the number of points $n\in\{20,\dots,300\}$ and computed the $L_2$-MZ constants of the results.
For each $n$, we started with $50$ different random initial draws of $n$ points and used the best result.
Here, also the weights are set equal to $\alpha_i = 1/n$ as the outcome did not differ compared to optimizing the weights as well.

\begin{figure} \centering
    \begingroup
  \makeatletter
  \providecommand\color[2][]{\GenericError{(gnuplot) \space\space\space\@spaces}{Package color not loaded in conjunction with
      terminal option `colourtext'}{See the gnuplot documentation for explanation.}{Either use 'blacktext' in gnuplot or load the package
      color.sty in LaTeX.}\renewcommand\color[2][]{}}\providecommand\includegraphics[2][]{\GenericError{(gnuplot) \space\space\space\@spaces}{Package graphicx or graphics not loaded}{See the gnuplot documentation for explanation.}{The gnuplot epslatex terminal needs graphicx.sty or graphics.sty.}\renewcommand\includegraphics[2][]{}}\providecommand\rotatebox[2]{#2}\@ifundefined{ifGPcolor}{\newif\ifGPcolor
    \GPcolortrue
  }{}\@ifundefined{ifGPblacktext}{\newif\ifGPblacktext
    \GPblacktexttrue
  }{}\let\gplgaddtomacro\g@addto@macro
\gdef\gplbacktext{}\gdef\gplfronttext{}\makeatother
  \ifGPblacktext
\def\colorrgb#1{}\def\colorgray#1{}\else
\ifGPcolor
      \def\colorrgb#1{\color[rgb]{#1}}\def\colorgray#1{\color[gray]{#1}}\expandafter\def\csname LTw\endcsname{\color{white}}\expandafter\def\csname LTb\endcsname{\color{black}}\expandafter\def\csname LTa\endcsname{\color{black}}\expandafter\def\csname LT0\endcsname{\color[rgb]{1,0,0}}\expandafter\def\csname LT1\endcsname{\color[rgb]{0,1,0}}\expandafter\def\csname LT2\endcsname{\color[rgb]{0,0,1}}\expandafter\def\csname LT3\endcsname{\color[rgb]{1,0,1}}\expandafter\def\csname LT4\endcsname{\color[rgb]{0,1,1}}\expandafter\def\csname LT5\endcsname{\color[rgb]{1,1,0}}\expandafter\def\csname LT6\endcsname{\color[rgb]{0,0,0}}\expandafter\def\csname LT7\endcsname{\color[rgb]{1,0.3,0}}\expandafter\def\csname LT8\endcsname{\color[rgb]{0.5,0.5,0.5}}\else
\def\colorrgb#1{\color{black}}\def\colorgray#1{\color[gray]{#1}}\expandafter\def\csname LTw\endcsname{\color{white}}\expandafter\def\csname LTb\endcsname{\color{black}}\expandafter\def\csname LTa\endcsname{\color{black}}\expandafter\def\csname LT0\endcsname{\color{black}}\expandafter\def\csname LT1\endcsname{\color{black}}\expandafter\def\csname LT2\endcsname{\color{black}}\expandafter\def\csname LT3\endcsname{\color{black}}\expandafter\def\csname LT4\endcsname{\color{black}}\expandafter\def\csname LT5\endcsname{\color{black}}\expandafter\def\csname LT6\endcsname{\color{black}}\expandafter\def\csname LT7\endcsname{\color{black}}\expandafter\def\csname LT8\endcsname{\color{black}}\fi
  \fi
    \setlength{\unitlength}{0.0500bp}\ifx\gptboxheight\undefined \newlength{\gptboxheight}\newlength{\gptboxwidth}\newsavebox{\gptboxtext}\fi \setlength{\fboxrule}{0.5pt}\setlength{\fboxsep}{1pt}\definecolor{tbcol}{rgb}{1,1,1}\begin{picture}(8220.00,3400.00)\gplgaddtomacro\gplbacktext{\csname LTb\endcsname \put(454,950){\makebox(0,0)[r]{\strut{}\scriptsize $10^{-12}$}}\csname LTb\endcsname \put(454,1835){\makebox(0,0)[r]{\strut{}\scriptsize $10^{-6}$}}\csname LTb\endcsname \put(454,2721){\makebox(0,0)[r]{\strut{}\scriptsize $10^{0}$}}\csname LTb\endcsname \put(566,450){\makebox(0,0){\strut{}\scriptsize 0}}\csname LTb\endcsname \put(1704,450){\makebox(0,0){\strut{}\scriptsize 100}}\csname LTb\endcsname \put(2843,450){\makebox(0,0){\strut{}\scriptsize 200}}\csname LTb\endcsname \put(3981,450){\makebox(0,0){\strut{}\scriptsize 300}}}\gplgaddtomacro\gplfronttext{\csname LTb\endcsname \put(2273,143){\makebox(0,0){\strut{}number of points}}\csname LTb\endcsname \put(2273,3072){\makebox(0,0){\strut{}random frequencies ($m=20$)}}}\gplgaddtomacro\gplbacktext{\csname LTb\endcsname \put(4554,1265){\makebox(0,0)[r]{\strut{}\scriptsize $10^{-12}$}}\csname LTb\endcsname \put(4554,1997){\makebox(0,0)[r]{\strut{}\scriptsize $10^{-6}$}}\csname LTb\endcsname \put(4554,2728){\makebox(0,0)[r]{\strut{}\scriptsize $10^{0}$}}\csname LTb\endcsname \put(4666,450){\makebox(0,0){\strut{}\scriptsize 0}}\csname LTb\endcsname \put(6374,450){\makebox(0,0){\strut{}\scriptsize 100}}\csname LTb\endcsname \put(8081,450){\makebox(0,0){\strut{}\scriptsize 200}}}\gplgaddtomacro\gplfronttext{\csname LTb\endcsname \put(6373,143){\makebox(0,0){\strut{}number of points}}\csname LTb\endcsname \put(6373,3072){\makebox(0,0){\strut{}$1$-dimensional bad frequencies ($m=10$)}}}\gplbacktext
    \put(0,0){\includegraphics[width={411.00bp},height={170.00bp}]{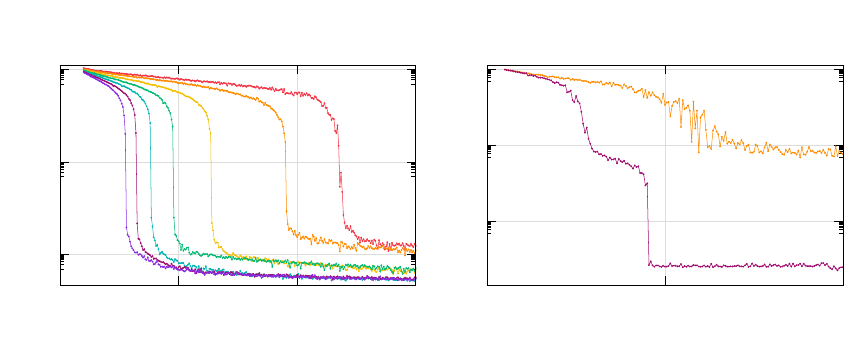}}\gplfronttext
  \end{picture}\endgroup
     \caption{
        Depiction of the $L_2$-MZ constant $\varepsilon = \|\sum_{i=1}^{N+1}\alpha_i \bm \varphi(\bm x^i)\cdot\bm \varphi(\bm x^i)^\ast - \bm I_m\|_{2\to 2}$ for:
        experiment~2 in {\color{red}$d=1$}, {\color{orange}$d=2$}, {\color{yellow}$d=3$}, {\color{green}$d=4$}, {\color{cyan}$d=5$}, {\color{indigo}$d=6$}, and {\color{violet}$d=7$} (left) and experiment~3 for $d=1$ {\color{indigo}with weights} and {\color{orange}without weights} (right).
    }\label{fig:trigonometric}
\end{figure}  

The results are depicted in the left of \Cref{fig:trigonometric}.
One can see that the threshold of finding a good $L_2$-MZ inequality is distinctly indicated by a sharp drop in the $L_2$-MZ constant $\varepsilon = \|\sum_{i=1}^{N+1}\alpha_i \bm \varphi(\bm x^i)\cdot\bm \varphi(\bm x^i)^\ast - \bm I_m\|_{2\to 2}$.
This happens around in accordance of our bound with $n=\dim\{V\cdot\overline V\} \le 381$.
Note, that for lower dimension this number may be smaller because of the fewer possibilities of differences.
In particular, we have $\dim\{V\cdot\overline V\}\approx 220$ for $d=1$, $\dim\{V\cdot\overline V\}\approx 375$ for $d=2$.
Furthermore, the results suggest a ``blessing of dimensionality'' as a good $L_2$-MZ inequality can be found with less points for higher dimensions.
In particular, the below table shows the first $n^\star(d)$ such that $\varepsilon < 10^{-10}$ with respect to the dimension $d$.
\begin{center}
    \setlength{\tabcolsep}{4pt}
    \begin{tabular}{c|*{20}{c}}
        $d$ & $1$ & $2$ & $3$ & $4$ & $5$ & $6$ & $7$ & $8$ & $9$ & $10$ & $11$ & $12$ & $13$ & $14$ & $15$ & $16$ & $17$ & $18$ & $19$ & $20$ \\
        \hline
        $n^\star(d)$ & $241$ & $192$ & $129$ & $97$ & $78$ & $66$ & $56$ & $49$ & $44$ & $39$ & $36$ & $33$ & $31$ & $29$ & $27$ & $25$ & $24$ & $23$ & $20$ & $20$ \\
        $\lfloor 400/d \rfloor$ & $400$ & $200$ & $133$ & $100$ & $80$ & $66$ & $57$ & $50$ & $44$ & $40$ & $36$ & $33$ & $30$ & $28$ & $26$ & $25$ & $23$ & $22$ & $21$ & $20$
    \end{tabular}
    \setlength{\tabcolsep}{6pt}
\end{center}
We added the values $\lfloor 400/d \rfloor$ which are close to $n^\star(d)$ suggesting an inverse linear dependence on the dimension.
A similar effect for the sphere $\mathds S^d$ is also discussed in \cite{GP11}, where it is argued that with higher dimension there are more degrees of freedom to be utilized.

\paragraph{Experiment~3} 

The third experiment aims to indicate the sharpness of our result.
For that we choose $d=1$ and $I = \{0$, $107\,062$, $124\,928$, $1\,033\,760$, $1\,414\,818$, $2\,142\,995$, $2\,820\,145$, $4\,210\,229$, $4\,645\,143$, $5\,264\,579\}$.
This is based on \Cref{cor:lattice_aliasing}, where we know that the minimal lattice size is $103$ for exact reconstruction and our theoretical bound is $\dim \{V\cdot\overline V\} = 91$.

Analogous to the second experiment, we computed the $L_2$-MZ constant varying the number of points $n\in\{20,\dots,200\}$.
For each $n$, we started with $1\,000$ different random initial draws of $n$ points and used the best result.
We repeated this experiment with and without optimizing the weights.

The results are seen in the right of \Cref{fig:trigonometric}.
The first observation is the improvement gained from utilizing non-equal weights.
It seems it makes the optimization procedure more stable as we know of the existence of an exact $L_2$-MZ inequality for $n=103$ which the equal-weighted algorithm did not find.
The weights in this case are non-equal.
The second observation is the visible drop to finding an exact $L_2$-MZ inequality at $n=91$, which is the same as our proposed bound for the needed number of points $\dim\{V\cdot\overline{V}\}$.
While this does not show the non-existence of possibly even equal-weighted exact $L_2$-MZ inequalities for $n\le 90$, it is a numerical indication of the sharpness of our result. \section{Guaranteed and verifiable recovery in $L_2(\mu)$}

Let us consider a recovery problem, which has been first addressed by Wasilkowski and Wo{\'z}niakowski \cite{WaWo01} in 2001 and drew a lot of attention in the past 6 years \cite{KrUL21,NaSchUl22,DoKrUl23,BSU23}. In this section we will comment on the problem how to practically find stable recovery algorithms and  suitable points which are guaranteed to recover any function from the unit ball in  a RKHS $H(k)$ with a prescribed accuracy in $L_2(\mu)$. In certain special case like Besov-Sobolev-spaces with dominating mixed smoothness, sparse grid and Smolyak type  algorithms serve as a practical near optimal way to recover such functions, see \cite[Chapt.\ 5]{DuTeUl18}. Other approaches include rank-$1$-lattice algorithms, see, e.g.,  \cite{Kammerer_Potts_Volkmer_2015} or \cite[Sect. 6]{KMNN21}, and versions thereof \cite{BaKaPoUl24}. The question arises whether it is possible to mimic such constructive approaches in the general setting of a reproducing kernel Hilbert space where we have access to the most important eigenfunctions of the embedding operator into $L_2(\mu)$. Our attempt will be to construct an exact $L_2$-discretization on the subspace spanned by these eigenfunctions via $D$-optimal designs. This is supposed to mimic the construction of a reproducing rank-$1$-lattice on spaces of trigonometric polynomials for given frequency sets in $\mathds Z^d$, like hyperbolic crosses, see \cite{Kae2012, Kae2013, KMNN21}. 

Given a probability measure $\mu$ on a compact topological space $\Omega$ and a bounded, continuous Mercer kernel $k(\cdot, \cdot):\Omega\times \Omega\to \mathds{C}$ we will construct a direct recovery algorithm for this specific situation using $D$-optimal designs. We start with the Mercer decomposition of the kernel $k(\cdot,\cdot)$ which is
\begin{equation}\label{SVD}
   k(x,y) = \sum\limits_{j=1}^\infty \sigma_j^2 \psi_j(x)\overline{\psi_j(y)}\,.
\end{equation}
Here the system $(\psi_j(\cdot))_j$ represents the system of orthonormal (w.r.t measure $\mu$) eigenfunctions of the corresponding integral operator with respect to the kernel $k(\cdot, \cdot)$ mapping from $L_2(\mu)$ to $L_2(\mu)$. The sequence $(\sigma_j^2)_{j \in \mathds{N}}$ represents the corresponding non-increasing sequence of non-negative eigenvalues satisfying 
$$
   \int_{\Omega}k(x,x)\,d\mu(x) = \sum_{j=1}^\infty \sigma_j^2 <\infty\,.
$$ 
Let us fix $n < \dim(H(k))$ and assume that we have access to the first $n$ (largest) singular values $\sigma_j$, the corresponding functions $\psi_j(\cdot)$ and the truncated trace
$$
    k_n(x,x) := \sum\limits_{j=1}^n \sigma_j^2 |\psi_j(x)|^2\,.
$$
We construct the sampling points and the linear sampling recovery operator $S^{k,\mu}_{n,N}: H(k) \to L_2(\mu)$ and start with the density function
\begin{equation}\label{eq333}
    \omega_n(x):= \frac{1}{2}+\frac{k(x,x)-k_n(x,x)}{2\int_\Omega (k(x,x)-k_n(x,x))\,d\mu(x)}=\frac{1}{2}+\frac{ \sum\limits_{j=n+1}^\infty \sigma_j^2 |\psi_j(x)|^2}{2\sum\limits_{j=n+1}^\infty \sigma_j^2}\,.
\end{equation}
Clearly, $\int_{\Omega} \omega_n(x) \;\mathrm d\mu(x) = 1$. We continue applying \Cref{discr_orth} to the modified system
\begin{equation*}
    \varphi_j(\cdot) := \frac{\psi_j(\cdot)}{\sqrt{\omega_n(\cdot)}}\quad,\quad j=1,...,n\,.
\end{equation*}
In particular, we apply \Cref{algo1} to the system $(\varphi_j(\cdot))_{j=1}^n$ and obtain points $(x_i)_{i=1}^N$ with $N\leq n^2+1$, and weights $(\lambda_i)_{i=1}^N$ summing up to $1$ such that the system matrix ${\bm A} = {\bm D}_{\bm \lambda}^{1/2}\cdot {\bm D}_{\bm \gamma}^{-1/2}\cdot {\bm \Psi}$ for the linear system
\begin{equation*}
    \begin{pmatrix}
        \sqrt{\frac{\lambda_1}{\omega_n(x^1)}}\psi_1(x^1) & \dots & \sqrt{\frac{\lambda_1}{\omega_n(x^1)}}\psi_n(x^1) \\
\vdots & \ddots & \vdots \\
\sqrt{\frac{\lambda_N}{\omega_n(x^N)}}\psi_1(x^N) & \dots & \sqrt{\frac{\lambda_N}{\omega_n(x^N)}}\psi_n(x^N) \\
    \end{pmatrix}
    \cdot
    \begin{pmatrix} c_1\\
    \vdots\\
    c_n
    \end{pmatrix}
    =
    \begin{pmatrix} \sqrt{\frac{\lambda_1}{\omega_n(x^1)}}f(x^1)\\
    \vdots\\
    \sqrt{\frac{\lambda_N}{\omega_n(x^N)}}f(x^N)
    \end{pmatrix}
\end{equation*}
has orthonormal columns. Introducing the matrices ${\bm \Psi}:=(\psi_k(x^j))_{j,k}$, ${\bm D}_{{\bm \lambda}} := \operatorname{diag}(\lambda_1,...,\lambda_n)$ and ${\bm D}_{\bm \omega} := \operatorname{diag}(\omega_n(x^1),...,\omega_n(x^N))$ we find
\begin{equation*}
    \bm c = ({\bm D}_{\bm \lambda}^{1/2}\cdot {\bm D}_{\bm \omega}^{-1/2}\cdot {\bm \Psi})^{\ast}\cdot {\bm D}_{\bm \lambda}^{1/2}\cdot {\bm D}_{\bm \omega}^{-1/2}\cdot {\bm f} = {\bm \Psi}^{\ast}\cdot{\bm D}_{\bm \lambda}\cdot {\bm D}^{-1}_{\bm \omega}\cdot{\bm f}\,,
\end{equation*}
where $\bm f = (f(x^1),...,f(x^N))^T = N(f)$, $\bm c = (c_1,\dots,c_n)^T$\,. The operator $S^{k,\mu}_{n,N}:H(K)\to L_2(\mu)$ is finally given by
\begin{equation*}
    S^{k,\mu}_{n,N} := E_n \circ {\bm A}^{\ast}\circ {\bm D}_{\bm \lambda}^{1/2}\circ {\bm D}_{\bm \omega}^{-1/2} \circ N
\end{equation*}
with $N~:~f \mapsto (f(x^1),...,f(x^N))^T$ and $E_n~:~\bm c \mapsto \sum_{j=1}^n c_j \psi_j$\,. This implies that the computation of $S^{k,\mu}_{n,N}f$ is possible in $O(n\cdot N)$ arithmetic operations. Note, that $N\leq n^2+1$, see \Cref{discr_orth}. 

In \cite{DoKrUl23} the authors improve on a result in \cite{NaSchUl22} and \cite{BSU23} and give an asymptotically sharp result in terms of the number of samples $n$ in cases where the singular values decay fast enough. It is shown that there is an algorithm $A_n$ which uses $N = 43200\cdot 866 \cdot n$ 
points such that for all $\|f\|_{H(k)} \leq 1$
$$
    \|f-A_n(f)\|_{L_2(\mu)}^2 \leq \frac{1}{n}\sum\limits_{j\geq n} \sigma_j^2\,.
$$
Despite its sharpness, the results in \cite{DoKrUl23,NaSchUl22} have two main drawbacks. Although the number of samples $N$ scales linearly in the dimension $n$ of the subspace, the oversampling constant may be huge. In addition, the point set is highly non-constructive and only its existence is proved. Also the non-optimal results in \cite{KrUL21,KaUlVo21,BSU23} propose points which are the  result of a random draw that can not be ``verified''.

The following theorem is a first ``verifiable'' attempt which guarantees the proposed accuracy. We refine an approach by Gr\"ochenig \cite{Gr20}, where error estimates for function classes are obtained from given Marcinkiewicz-Zygmund families. Here we use the exact ones from \Cref{Sect:discrete_orth} which lead to direct algorithms.

\begin{theorem} Let $\Omega$ be a compact topological measurable space with a Borel probability measure $\mu$ and $k(\cdot, \cdot):\Omega\times \Omega\to \mathds{C}$ a Mercer kernel. Then the sampling recovery algorithm $S^{k,\mu}_{n,N} : H(k) \to L_2(\mu)$ defined above yields the recovery bound
    \begin{equation}\label{fast}
        \sup\limits_{\|f\|_{H(k)\leq 1}}\|f-S^{k,\mu}_{n,N} f\|_{L_2(\mu)}^2 \leq 3\sum\limits_{j\geq n+1} \sigma_j^2\leq 3\,c_n(H(k),C(\Omega))^2\,,
    \end{equation}
    where $c_n(H(k),C(\Omega))$ denotes the $n$th Gelfand number of the identity operator from $H(k)$ into $C(\Omega)$.
    In addition, $S^{k,\mu}_{n,N}f$ uses  $N\leq n^2+1$ many function samples and a value $S^{k,\mu}_{n,N}f(x)$ can be directly and accurately computed with less than $c\cdot n^3$ arithmetic operations, where $c>0$ is an absolute constant.
\end{theorem} 

\begin{proof}  Let the projection operator $A_n:H(k) \to L_2(\mu)$ be given as
    \begin{equation*}
        A_nf := \sum\limits_{j=1}^n \langle f, \sigma_j \psi_j\rangle_{H(k)}\sigma_j\psi_j(\cdot)\,.
    \end{equation*}
    By a straight-forward computation we get for all $x\in \Omega$
    \begin{equation}\label{eq334}
        |f(x)-A_nf(x)|^2 \leq \|f\|_{H(k)}^2\cdot \sum\limits_{j=n+1}^\infty \sigma_j^2 |\psi_j(x)|^2
    \end{equation}
    but also $\|f-A_nf\|^2_{L_2(\mu)} \leq \sigma_{n+1}^2\|f\|^2_{H(k)}$ which will be used in the following estimation. Fix $f \in H(k)$ with $\|f\|_{H(k)} \leq 1$. By $L_2(\mu)$-orthogonality we have
    \begin{equation}
       \begin{split}
            \|f-S^{k,\mu}_{n,N}f\|_{L_2(\mu)}^2 &= \|f-A_nf\|_{L_2(\mu)}^2 + \|A_nf - S^{k,\mu}_{n,N}f\|_{L_2(\mu)}^2\\
            &=\sigma_{n+1}^2+\|S^{k,\mu}_{n,N}(A_nf - f)\|_{L_2(\mu)}^2\\
            &=\sigma_{n+1}^2+\|E_n\circ{\bm A}^{\ast}\circ {\bm D}_{\bm \lambda}^{1/2}\circ {\bm D}_{\bm \omega}^{-1/2}\circ N(A_nf-f)\|^2_{L_2(\mu)}\\
            &=\sigma_{n+1}^2 +\Big\|{\bm A}^{\ast}\circ {\bm D}_{\bm \lambda}^{1/2}\circ {\bm D}_{\bm \omega}^{-1/2}\circ N(A_nf-f)\Big\|^2_{\ell_2^n}\\
            &\leq \sigma_{n+1}^2 +\Big\|{\bm D}_{\bm \lambda}^{1/2}\circ {\bm D}_{\bm \omega}^{-1/2}\circ N(A_nf-f)\Big\|^2_{\ell_2^N}\\
            &\leq \sigma_{n+1}^2 +\sum\limits_{i=1}^N\lambda_i\frac{|f(x^i)-A_nf(x^i)|^2}{\omega_n(x^i)}\\
\end{split}
    \end{equation}
    In the last but one estimate we used that $\|{\bm A}^{\ast}\|_{2\to 2} = \|\bm A\|_{2\to 2}  = 1$. Taking \eqref{eq333} and \eqref{eq334} into account yields
    \begin{equation*}
        \|f-S^{k,\mu}_{n,N}f\|_{L_2(\mu)}^2 \leq \sigma_{n+1}^2+2\sum\limits_{j= n+1}^{\infty} \sigma_j^2\sum\limits_{i=1}^N\lambda_i \leq 3\sum\limits_{j=n+1}^\infty \sigma_j^2\,.
    \end{equation*}
    The second inequality in \eqref{fast} follows from \cite[Lem.\ 3.3]{CoKuSi16} and the fact that Gelfand and approximation numbers coincide in our situation, see \cite[Thm.\ 4.8]{NoWo08}.   
\end{proof} 

\begin{remark}[Point cloud constructions] Recent approaches for obtaining point clouds involve an initial random draw, see \cite{KrUL21,NaSchUl22,BSU23,DoKrUl23}. Our approach characterizes the point set via a maximum search for determinants / Frobenius norms. In special cases, i.e., rank-$1$-lattices, one may prove better results \cite{ByKaUlVo17}, \cite{BaKaPoUl24}. In addition, so-called sparse grids yield another constructive approach for suitable point clouds which work particularly well in the special case of mixed smoothness RKHS, see \cite{DuTeUl18} and \cite{SiUl07}. We may further BSS-subsample \cite{BSU23} the points obtained in the optimization process to get a similar bound with fewer points like in \cite{KPUU24}. The right-hand side in \eqref{fast} keeps valid up to constants, whereas the number of point samples shrinks to $N = O(n)$. The subsampling is constructive.
\end{remark} 

\begin{remark}[Direct and stable recovery algorithms] Using points from an exact $L_2$-MZ inequality gives a direct and stable method to obtain the approximation, i.e., only a multiplication with the perfectly conditioned adjoint system matrix $\bm A$ is needed rather than a matrix inversion. With the naive matrix-vector multiplication, this yields a computational complexity of $\mathcal O(n\cdot N) = \mathcal O(n^3)$ when using $n$ basis functions. In contrast to that, current subsampling techniques yield $N = O(n)$ many points with a near-optimal  error behavior with respect to the number of points \cite{BSU23,BaKaPoUl24}.
    These points do not fulfill an exact $L_2$-MZ inequality and one usually computes the approximation by (weighted) least squares algorithms.
    With $r$ iterations, this gives a computational complexity of $\mathcal O(r\cdot n\cdot N) = \mathcal O(r\cdot n^2)$.
    For an exact solution $r=n$ iterations could be necessary which yields also a computational complexity of $\mathcal O(n^3)$.
    However, the error decay is exponential in $r$ and often only $20$ iterations are sufficient, cf.\ \cite[Thm.~3.1.1]{Gre97}, which would yield a computational complexity of $\mathcal O(n^2)$.
\end{remark} 

\noindent\textbf{Acknowledgement.} KP would like to acknowledge support by the Philipp Schwartz Fellowship of the Alexander von
Humboldt Foundation. TU and FB acknowledge support by the European Union, European Social Fund ESF-Plus, Saxony, where this work was part of the research project ReSIDA-H2. LK gratefully acknowledges funding by the Deutsche Forschungsgemeinschaft (DFG,
German Research Foundation) – project number 380648269. 
\bibliographystyle{abbrv}

\begin{thebibliography}{10}

\bibitem{ACSW10}
C.~An, X.~Chen, I.~H. Sloan, and R.~S. Womersley.
\newblock Well conditioned spherical designs for integration and interpolation
  on the two-sphere.
\newblock {\em SIAM J. Numer. Anal.}, 48(6):2135–2157, Jan. 2010.

\bibitem{BaKaPoUl24}
F.~Bartel, L.~K\"{a}mmerer, D.~Potts, and T.~Ullrich.
\newblock On the reconstruction of functions from values at subsampled
  quadrature points.
\newblock {\em Math. Comp.}, 93(346):785--809, 2024.

\bibitem{BSU23}
F.~Bartel, M.~Sch\"{a}fer, and T.~Ullrich.
\newblock Constructive subsampling of finite frames with applications in
  optimal function recovery.
\newblock {\em Appl. Comput. Harmon. Anal.}, 65:209--248, 2023.

\bibitem{BaSpSr12}
J.~Batson, D.~A. Spielman, and N.~Srivastava.
\newblock Twice-{R}amanujan sparsifiers.
\newblock {\em SIAM J. Comput.}, 41(6):1704--1721, 2012.

\bibitem{BRV13}
A.~Bondarenko, D.~Radchenko, and M.~Viazovska.
\newblock Optimal asymptotic bounds for spherical designs.
\newblock {\em Annals of Mathematics}, 178(2):443--452, 2013.

\bibitem{Bo22}
L.~Bos.
\newblock On optimal designs for a $d$-cube.
\newblock {\em Dolomites Research Notes on Approximation}, 15(4):20--34, 2022.

\bibitem{Bo23}
L.~Bos.
\newblock On {F}ekete points for a real simplex.
\newblock {\em Indagationes Mathematicae}, 34(2):274--293, 2023.
\newblock Special Issue on the occasion of Jaap Korevaar’s 100-th
  anniversary.

\bibitem{BoPiVi20}
L.~Bos, F.~Piazzon, and M.~Vianello.
\newblock Near {G}-optimal {T}chakaloff designs.
\newblock {\em Comput. Statist.}, 35(2):803--819, 2020.

\bibitem{ByKaUlVo17}
G.~Byrenheid, L.~K\"{a}mmerer, T.~Ullrich, and T.~Volkmer.
\newblock Tight error bounds for rank-1 lattice sampling in spaces of hybrid
  mixed smoothness.
\newblock {\em Numer. Math.}, 136(4):993--1034, 2017.

\bibitem{Cahill_Chen_2013}
J.~Cahill and X.~Chen.
\newblock A note on scalable frames.
\newblock In {\em Conference proceedings of SampTA}. Zenodo, 2013.

\bibitem{Casazza_Chen_2017}
P.~G. Casazza and X.~Chen.
\newblock Frame scalings: A condition number approach.
\newblock {\em Linear Algebra and its Applications}, 523:152--168, 2017.

\bibitem{CFM12}
P.~G. Casazza, M.~Fickus, and D.~G. Mixon.
\newblock Auto-tuning unit norm frames.
\newblock {\em Applied and Computational Harmonic Analysis}, 32(1):1–15, Jan.
  2012.

\bibitem{CK03}
P.~G. Casazza and J.~Kovačević.
\newblock Equal-norm tight frames with erasures.
\newblock {\em Advances in Computational Mathematics}, 18(2/4):387–430, 2003.

\bibitem{CoKuSi16}
F.~Cobos, T.~K\"{u}hn, and W.~Sickel.
\newblock Optimal approximation of multivariate periodic {S}obolev functions in
  the sup-norm.
\newblock {\em J. Funct. Anal.}, 270(11):4196--4212, 2016.

\bibitem{DuTeUl18}
D.~D\~{u}ng, V.~Temlyakov, and T.~Ullrich.
\newblock {\em Hyperbolic cross approximation}.
\newblock Advanced Courses in Mathematics. CRM Barcelona.
  Birkh\"{a}user/Springer, Cham, 2018.
\newblock Edited and with a foreword by Sergey Tikhonov.

\bibitem{DaPrTeTi19}
F.~Dai, A.~Prymak, V.~N. Temlyakov, and S.~Y. Tikhonov.
\newblock Integral norm discretization and related problems.
\newblock {\em Russian Mathematical Surveys}, 74(4):579--630, 2019.

\bibitem{DX13}
F.~Dai and Y.~Xu.
\newblock {\em Approximation Theory and Harmonic Analysis on Spheres and
  Balls}.
\newblock Springer New York, 2013.

\bibitem{Delsarte_Goethals_Seidel1977}
P.~Delsarte, J.~M. Goethals, and J.~J. Seidel.
\newblock Spherical codes and designs.
\newblock {\em Geometriae Dedicata}, 6(3):363--388, 1977.

\bibitem{DWD22}
M.~Derezi\'{n}ski, M.~K. Warmuth, and D.~Hsu.
\newblock Unbiased estimators for random design regression.
\newblock {\em J. Mach. Learn. Res.}, 23:Paper No. [167], 46, 2022.

\bibitem{DeSt97}
H.~Dette and W.~J. Studden.
\newblock {\em The theory of canonical moments with applications in statistics,
  probability, and analysis}.
\newblock Wiley Series in Probability and Statistics: Applied Probability and
  Statistics. John Wiley \& Sons, Inc., New York, 1997.
\newblock A Wiley-Interscience Publication.

\bibitem{DoKrUl23}
M.~Dolbeault, D.~Krieg, and M.~Ullrich.
\newblock A sharp upper bound for sampling numbers in {$L_2$}.
\newblock {\em Appl. Comput. Harmon. Anal.}, 63:113--134, 2023.

\bibitem{FiHiJaUl24}
F.~Filbir, R.~Hielscher, T.~Jahn, and T.~Ullrich.
\newblock Marcinkiewicz-{Z}ygmund inequalities for scattered and random data on
  the {$q$}-sphere.
\newblock {\em Appl. Comput. Harmon. Anal.}, 71:Paper No. 101651, 18, 2024.

\bibitem{FREEMAN2023126846}
D.~Freeman and D.~Ghoreishi.
\newblock {Discretizing $L_p$ norms and frame theory}.
\newblock {\em Journal of Mathematical Analysis and Applications},
  519(2):126846, 2023.

\bibitem{FS19}
D.~Freeman and D.~Speegle.
\newblock The discretization problem for continuous frames.
\newblock {\em Adv. Math.}, 345:784--813, 2019.

\bibitem{GP11}
M.~Gr\"{a}f and D.~Potts.
\newblock {On the computation of spherical designs by a new optimization
  approach based on fast spherical Fourier transforms}.
\newblock {\em Numerische Mathematik}, 119(4):699–724, July 2011.

\bibitem{Gre97}
A.~Greenbaum.
\newblock {\em Iterative methods for solving linear systems}, volume~17 of {\em
  Frontiers in Applied Mathematics}.
\newblock Society for Industrial and Applied Mathematics (SIAM), Philadelphia,
  PA, 1997.

\bibitem{Gr20}
K.~Gröchenig.
\newblock Sampling, {Marcinkiewicz–Zygmund} inequalities, approximation, and
  quadrature rules.
\newblock {\em Journ. Approx. Theory}, 257:105455, 2020.

\bibitem{Kae2012}
L.~K{\"a}mmerer.
\newblock Reconstructing hyperbolic cross trigonometric polynomials by sampling
  along rank-1 lattices.
\newblock {\em SIAM J. Numer. Anal.}, 51:2773--2796, 2013.

\bibitem{Kae2013}
L.~K{\"a}mmerer.
\newblock Reconstructing multivariate trigonometric polynomials from samples
  along rank-1 lattices.
\newblock In G.~E. Fasshauer and L.~L. Schumaker, editors, {\em Approximation
  Theory XIV: San Antonio 2013}, pages 255--271. Springer International
  Publishing, 2014.

\bibitem{Kammerer_Potts_Volkmer_2015}
L.~K\"{a}mmerer, D.~Potts, and T.~Volkmer.
\newblock Approximation of multivariate periodic functions by trigonometric
  polynomials based on sampling along rank-1 lattice with generating vector of
  {K}orobov form.
\newblock {\em Journal of Complexity}, 31(3):424--456, 2015.

\bibitem{KaUlVo21}
L.~K\"{a}mmerer, T.~Ullrich, and T.~Volkmer.
\newblock Worst-case recovery guarantees for least squares approximation using
  random samples.
\newblock {\em Constr. Approx.}, 54(2):295--352, 2021.

\bibitem{KW60}
J.~Kiefer and J.~Wolfowitz.
\newblock The equivalence of two extremum problems.
\newblock {\em Canadian J. Math.}, 12:363--366, 1960.

\bibitem{Ko22}
E.~D. Kosov.
\newblock Remarks on sampling discretization of integral norms of functions.
\newblock {\em Tr. Mat. Inst. Steklova}, 319:202--212, 2022.

\bibitem{KPUU24}
D.~Krieg, K.~Pozharska, M.~Ullrich, and T.~Ullrich.
\newblock Sampling projections in the uniform norm.
\newblock {\em arXiv:math/2401.02220}, 2024.

\bibitem{KrUL21}
D.~Krieg and M.~Ullrich.
\newblock Function values are enough for {$L_2$}-approximation.
\newblock {\em Found. Comput. Math.}, 21(4):1141--1151, 2021.

\bibitem{KMNN21}
F.~Y. Kuo, G.~Migliorati, F.~Nobile, and D.~Nuyens.
\newblock Function integration, reconstruction and approximation using rank-$1$
  lattices.
\newblock {\em Math. Comput.}, 90(330):1861–1897, Apr. 2021.

\bibitem{Kutyniok_Okoudjou_Philipp}
G.~Kutyniok, K.~Okoudjou, and F.~Philipp.
\newblock Scalable frames and convex geometry.
\newblock {\em Contemporary Mathematics}, 626:19–32, 2014.

\bibitem{Marcinkiewicz_Zygmund_1937}
J.~Marcinkiewicz and A.~Zygmund.
\newblock Sur les fonctions indépendantes.
\newblock {\em Fundamenta Mathematicae}, 29(1):60--90, 1937.

\bibitem{MaSpSr15}
A.~W. Marcus, D.~A. Spielman, and N.~Srivastava.
\newblock Interlacing families {II}: {M}ixed characteristic polynomials and the
  {K}adison-{S}inger problem.
\newblock {\em Ann. of Math. (2)}, 182(1):327--350, 2015.

\bibitem{NaSchUl22}
N.~Nagel, M.~Sch\"{a}fer, and T.~Ullrich.
\newblock A new upper bound for sampling numbers.
\newblock {\em Found. Comput. Math.}, 22(2):445--468, 2022.

\bibitem{NiOlUl16}
S.~Nitzan, A.~Olevskii, and A.~Ulanovskii.
\newblock Exponential frames on unbounded sets.
\newblock {\em Proc. Amer. Math. Soc.}, 144(1):109--118, 2016.

\bibitem{Nov86}
E.~Novak.
\newblock Quadrature and widths.
\newblock {\em Journal of Approximation Theory}, 47(3):195--202, 1986.

\bibitem{NoWo08}
E.~Novak and H.~Wo\'{z}niakowski.
\newblock {\em Tractability of multivariate problems. {V}ol. 1: {L}inear
  information}, volume~6 of {\em EMS Tracts in Mathematics}.
\newblock European Mathematical Society (EMS), Z\"{u}rich, 2008.

\bibitem{PiSoVi17}
F.~Piazzon, A.~Sommariva, and M.~Vianello.
\newblock Caratheodory-{T}chakaloff subsampling.
\newblock {\em Dolomites Res. Notes Approx.}, 10(1):5--14, 2017.

\bibitem{Pu97}
M.~Putinar.
\newblock A note on {T}chakaloff's theorem.
\newblock {\em Proc. Amer. Math. Soc.}, 125(8):2409--2414, 1997.

\bibitem{Schn14}
R.~Schneider.
\newblock {\em Convex bodies: the {B}runn-{M}inkowski theory}, volume 151 of
  {\em Encyclopedia of Mathematics and its Applications}.
\newblock Cambridge University Press, Cambridge, expanded edition, 2014.

\bibitem{Sha71}
H.~S. Shapiro.
\newblock {\em Topics in approximation theory}.
\newblock Lecture Notes in Mathematics, Vol. 187. Springer-Verlag, Berlin-New
  York, 1971.
\newblock With appendices by Jan Boman and Torbj\"{o}rn Hedberg.

\bibitem{SiUl07}
W.~Sickel and T.~Ullrich.
\newblock The {S}molyak algorithm, sampling on sparse grids and function spaces
  of dominating mixed smoothness.
\newblock {\em East J. Approx.}, 13(4):387--425, 2007.

\bibitem{SlLy89}
I.~H. Sloan and J.~N. Lyness.
\newblock The representation of lattice quadrature rules as multiple sums.
\newblock {\em Mathematics of Computation}, 52(185):81--94, 1989.

\bibitem{Tschak57}
V.~Tchakaloff.
\newblock Formules de cubatures m\'{e}caniques \`a coefficients non
  n\'{e}gatifs.
\newblock {\em Bull. Sci. Math. (2)}, 81:123--134, 1957.

\bibitem{Te17}
V.~N. Temlyakov.
\newblock The {M}arcinkiewicz-type discretization theorems for the hyperbolic
  cross polynomials.
\newblock {\em Jaen J. Approx.}, 9(1):37--63, 2017.

\bibitem{Tropp_2011}
J.~A. Tropp.
\newblock User-friendly tail bounds for sums of random matrices.
\newblock {\em Foundations of Computational Mathematics}, 12(4):389--434, 2012.

\bibitem{WaWo01}
G.~W. Wasilkowski and H.~Wo\'{z}niakowski.
\newblock On the power of standard information for weighted approximation.
\newblock {\em Found. Comput. Math.}, 1(4):417--434, 2001.

\bibitem{Womersley18}
R.~S. Womersley.
\newblock {\em Efficient Spherical Designs with Good Geometric Properties},
  page 1243–1285.
\newblock Springer International Publishing, 2018.

\end{thebibliography}

\end{document}